\documentclass{article}
\usepackage[a4paper, margin=1in]{geometry}
\usepackage{style}
\usepackage{graphicx}  
\usepackage{booktabs} 

\title{The Convex Matching Distance in Multiparameter Persistence}

\author{
Francesco Conti
\and
Patrizio Frosini
\and
Ulderico Fugacci
\and
Eloy Mosig Garc\'\i a
\and
Nicola Quercioli
\and
Sara Scaramuccia
\and
Francesca Tombari}

\date{}

\begin{document}

\maketitle

\begin{abstract}
We introduce the convex matching distance, a novel metric for comparing functions with values in the real plane. This metric measures the maximal bottleneck distance between the persistence diagrams associated with the convex combinations of the two function components. Similarly to the traditional matching distance, the convex matching distance aggregates the information provided by two real-valued components. However, whereas the matching distance depends on two parameters, the convex matching distance depends on only one, offering improved computational efficiency. We further show that the convex matching distance can be more discriminative than the traditional matching distance in certain cases, although the two metrics are generally not comparable. Moreover, we prove that the convex matching distance is stable and characterize the coefficients of the convex combination at which it is attained. Finally, we demonstrate that this new aggregation framework benefits from the computational advantages provided by the Pareto grid, a collection of curves in the plane whose points lie in the image of the Pareto critical set associated with functions assuming values on the real plane.
Experimental validation on MNIST digits, synthetic shapes, and chaotic attractors suggests that the convex matching distance provides a reliable and efficient alternative to the matching distance, at a significantly lower computational cost.
\end{abstract}

\subsection*{Keywords}
Biparameter persistent homology, Matching distance, Pareto grid.


\section{Introduction}



In recent years, research on multiparameter persistent homology has grown substantially, driven by the need to apply Topological Data Analysis (TDA) techniques to data represented by multifiltrations, as occurs, for example, when analysing or comparing data described by vector-valued functions \cite{FrMu99,BiCeFrGiLa08,CaZo09,CaDFFe10,CaSiZo10,CaLa11,CaEtFrKaLa13,CeLa13,CaFeLa15,CeFr15,Le15,AlKaLa17,HaOtScTi19}. However, a fully developed theoretical framework enabling efficient applications of TDA in this setting is still lacking.
A widely adopted comparison method in multiparameter persistence is based on the \textit{matching distance}
\cite{d’Amico2010,BiCeFrGi11,CeDFFeFrLa13,KeLeOu20,Landi2022,BjKe23}.
This pseudo-metric is introduced by observing that, given a continuous function $\boldsymbol{\p}:X\to \R^n$
representing the data of interest, we can consider monoparametric filtrations of the topological space $X$ associated with lines in $\R^n$ having direction given by a vector $\boldsymbol{{a}}=({a}_1,\ldots,{a}_n)$ such that $\sum_{i=1}^n a_i=1$ and $a_1,\ldots,a_n > 0$.
Each such line, parametrised as $\boldsymbol{b}+u\boldsymbol{{a}}$ with $\boldsymbol{b}=(b_1,\ldots,b_n)$ and $\sum_{i=1}^n b_i=0$, induces a filtration $\{X^{(\bo a,\bo b)}_u\}_{u\in \R}$ of $X$ defined by the subsets $X^{(\bo a,\bo b)}_u$
consisting of all points in $X$ whose image under $\boldsymbol{\p}$ has coordinates less than or equal to those of $\boldsymbol{b}+u\boldsymbol{{a}}$.
From an analytical viewpoint, this construction is equivalent to considering, for each function $\boldsymbol{\varphi}=(\p_1,\ldots,\p_n):X\to\R^n$,
the sublevel sets of the real-valued function
\begin{equation}
\label{eq:intro}
    {\boldsymbol{\p}}^*_{\boldsymbol{{a}},\boldsymbol{b}}(p):=\min_{1\le i\le n}\left\{{a}_i\right\}\max_{1\le i\le n}\left\{\frac{\p_i(p)-b_i}{{a}_i}\right\}.
\end{equation}
With this premise, the matching distance in degree $k$ between two continuous functions $\boldsymbol{\p},\boldsymbol{\psi}:X\to\R^n$ is defined by setting
$$d_{\mathrm{match},k}(\boldsymbol{\p},\boldsymbol{\psi}):=
\sup_{{\boldsymbol{{a}},\boldsymbol{b}}}
d_B\left(\dgm_k\left({\boldsymbol{\p}}^*_{\boldsymbol{{a}},\boldsymbol{b}}\right),\dgm_k\left({\boldsymbol{\psi}}^*_{\boldsymbol{{a}},\boldsymbol{b}}\right)\right),$$
where $d_B$ is the bottleneck distance and $\dgm_k$ is the operator that computes the persistence diagrams in degree $k$.

In recent years, the extended Pareto Grid has been introduced to compute points in the persistence diagram of $\bo\p^*_{\bo a, \bo b}$, and to characterise at which $(\bo a, \bo b)$ the matching distance is attained \cite{coher_match,EFQT2023,FrMoQuTo25}.

While the study of the matching distance is geometrically interesting and theoretically rich
it involves $2(n-1)$ independent parameters and depends non-smoothly on the function $\boldsymbol{\p}$, leading to
computational challenges and monodromy \cite{coher_match}.
We address this issue by considering the operator that maps a function
\[
\boldsymbol{\varphi} = (\varphi_1, \ldots, \varphi_n) \colon X \to \mathbb{R}^n
\]
to a family of real-valued functions $\boldsymbol{\varphi}^{\boldsymbol{t}}$ defined as follows:
for each $\boldsymbol{t} = (t_1, \ldots, t_n) \in \mathbb{R}^n$ satisfying $\sum_{i=1}^n t_i = 1$ and $t_i \ge 0$ for all $i$, we set
\[
\boldsymbol{\varphi}^{\boldsymbol{t}} := \sum_{i=1}^n t_i \varphi_i.
\]
The conditions on $\boldsymbol{t}$ imply $t_n = 1 - \sum_{i=1}^{n-1} t_i$; hence, the family
$\{\boldsymbol{\varphi}^{\boldsymbol{t}}\}_{\boldsymbol{t} \in \R^n}$ depends on $n-1$ parameters, substantially reducing the complexity of the approach based on the parameters $\boldsymbol{a}$ and $\boldsymbol{b}$.
Moreover, our operator $F_{\boldsymbol{t}}(\bo\p)=\bo\p^{\boldsymbol{t}}$ depends smoothly on the parameters $\boldsymbol{t}$ and the function $\bo\p$, in contrast with the operator in~\eqref{eq:intro}.

Furthermore, the operator $F_{\boldsymbol{t}}$ gives rise to a new pseudo-metric between filtering functions, called the \emph{convex matching distance}, defined as
\[
\cmd_k(\boldsymbol{\varphi}, \boldsymbol{\psi})
:=
\max_{\substack{
    t_1,\ldots,t_n\ge 0\\
    t_1+\ldots+t_n = 1
}}
d_B\!\left(
    \dgm_k(\boldsymbol{\varphi}^{\boldsymbol{t}}),
    \dgm_k(\boldsymbol{\psi}^{\boldsymbol{t}})
\right).
\]
Although this pseudo-metric is not always more discriminative than the classical matching distance (as shown by Examples~\ref{exKC1} and~\ref{exKC2}), its dependence on $n-1$ parameters, together with the fact that it is defined via a smooth procedure, provides significant computational advantages from an applied perspective.

The theoretical framework from which the convex matching distance derives is that of
\emph{group equivariant non-expansive operators (GENEOs)}, which were introduced ten years ago
to restrict the invariance of persistence diagrams in Topological Data Analysis \cite{FrJa16}.
In particular, the convex matching distance originates from the paper \cite{bergomi2019towards},
where it was proved that selecting an arbitrary family of GENEOs leads to the construction of
new stable matching distances (Proposition~9 and Theorem~10; see also \cite{Fr16}). In the present paper, we focus on the family of operators
\(F_{\boldsymbol{t}}(\bo \p)=\bo \p^{\boldsymbol{t}}\),
which yields the convex matching distance, owing to the interesting computational properties of
this metric.
Interestingly, infinitely many other examples of parametric families of GENEOs can be defined. For instance, restricting ourselves to monoparametric families, we may consider the following operators:
\[
\hat{F}_t(\bo \p)
:= (1-t)\max(\p_1, \p_2) + t\min(\p_1, \p_2),
\quad t \in [0,1],
\]
and
\[
\tilde{F}_t(\bo \p)
:=
\left(\tfrac{1}{2}|\p_1|^t + \tfrac{1}{2}|\p_2|^t\right)^{\!\frac{1}{t}},
\quad t \ge 1.
\]

Each of these operators could be used to define new pseudo-metrics in TDA.
We emphasize that the theory of GENEOs also has interesting applications beyond classical Topological Data Analysis
(cf., e.g., \cite{bergomi2019towards,Mi23,FaFeFr2023,BoBoBrFrQu23,BocchiFMPPGGLBF24,
CoBoGiGiPeFr25,Felipe2025,BoFeFr25,BoFrMi25,LaMiBoFrSo25,Ah26}).

For alternative approaches to this idea of a unifying theory to generate distances in persistence, we refer to~\cite{BePe26} where projected distances between multiparameter persistence modules are introduced in the context of sheaf theoretic persistence.
The convex matching distance can also be framed in this context.

The approach described in this paper is related to, but distinct from, the \emph{persistent homology transform} (PHT) \cite{TuMuBo14}.
The PHT is a topological transform that takes as input a subset of a Euclidean space and associates to every unit vector the collection of persistence diagrams of the height functions on that subset in the direction given by the vector.
A distance between two subsets is then defined by integrating over the sphere the distances between the respective persistence diagrams.
The convex matching distance, instead, is obtained as a maximum over the unit vectors of non-negative slope, thus, over a smaller set.
Considering only vectors of positive slope is a crucial aspect of our techniques because, together with some regularity assumptions, it enables the use of a convenient differential geometry tool, called the \emph{Pareto grid}.
Another key difference between the two approaches lies in their scope.
While the PHT is designed to recover the geometry of the object $X$, our method, based on convex combinations of a vector-valued function, focuses on making inferences about the function whose domain is $X$.

In this paper, we study the convex matching distance in the case $n = 2$.
We prove that the convex matching distance is stable with respect to the uniform norm $\|\boldsymbol{\varphi} - \boldsymbol{\psi}\|_\infty$ (Proposition~\ref{pro:stable}) and establish a \emph{Position Theorem} (Theorem~\ref{thm:position}), analogous to the one already known for the matching distance~\cite{coher_match,FrMoQuTo25}.
This theorem allows one to recover, from the Pareto grid of $\boldsymbol{\varphi}$, the coordinates of the points in the persistence diagram of $\boldsymbol{\varphi}^{t} := (1-t)\p_1+t\p_2$, for every $t \in [0,1]$.
Finally, in Theorem~\ref{thm:main}, we show that the convex matching distance coincides with the maximal bottleneck distance between the persistence diagrams of the functions $\boldsymbol{\varphi}^{t}$ and $\boldsymbol{\psi}^{t}$ for $t$ belonging to a suitable set of special values,
thus paving the way for an efficient method to compute the convex matching distance.
We validate the practical effectiveness of the convex matching distance through experiments on three benchmark datasets: MNIST digit classification, synthetic 3D shapes under varying noise levels, and chaotic dynamical attractors.
Our results highlight that the convex matching distance achieves comparable classification performance to the matching distance computed on a fine $11 \times 11$ parameter grid, while requiring only $11$ persistence diagram computations instead of $121$.
Moreover, we observe that the convex matching distance can outperform coarser approximations of the matching distance, confirming that the convex parametrization captures the essential discriminative information of the bifiltration.
These findings establish the convex matching distance as a computationally efficient and theoretically grounded alternative to existing multiparameter persistence metrics.

\section{Preliminaries}\label{sec:setting}

In this section, we recall the bottleneck distance in monoparameter persistence and the matching distance in biparameter persistence.
To do that, some preliminary notions about persistence theory is needed, and certain hypothesis need to be assumed.
However, as this is fairly standard, we will not delve into details, and refer the reader to~\cite[Section 2 and Appendix A]{FrMoQuTo25} for a compact treatment of preliminaries.
The following is a complete list of symbols used in this article
\footnote{Note that these symbols might differ slightly from those used in \cite{FrMoQuTo25}.}:

\begin{itemize}
\item $X$ is a compact topological space;
\item $\Delta:=\{(u, u)\mid u\in \R\}$, $\Delta^+:=\{(u,v)\mid u,v\in \R, u<v\}$, $\bar \Delta^*:=\Delta^+\cup \{(u,\infty)\mid u\in \R\}\cup \{\Delta\}$;
\item $\dgm_k(\p)$ is the $k$-th persistence diagram of the continuous function $\p\colon X\to \R$, and it is a multiset in $\bar\Delta^*$, where each of its points has a multiplicity, and, in particular, $\Delta$ is a point in $\dgm_k(\p)$ with infinite multiplicity;
\item the following function $d\colon \bar\Delta^*\times \bar \Delta^*\to [0,\infty]$ is an extended metric
\[d(p,q) := \begin{cases}
        C(u,u',v,v') \quad &\text{if } {p}=(u,v), {q}=(u',v') \in \Delta^+,\\
        \vert u-u'\vert &\text{if } {p}=(u,\infty), {q}=(u',\infty),\\
        \frac{v-u}{2} &\text{if } {p}=(u,v) \in \Delta^+, =\Delta,\\
        \frac{v'-u'}{2} &\text{if } {p} = \Delta, {q}=(u',v') \in \Delta^+,\\
        0 &\text{if } {p} = {q} =\Delta,\\
        \infty &\text{otherwise,}
    \end{cases}\]
    where $C(u,u',v,v') := \min\{\max\{\vert u-u'\vert,\vert v-v'\vert\},\max\{\frac{v-u}{2},\frac{v'-u'}{2}\}\}$;
\item if $\boldsymbol{\p}=(\p_1,\p_2)\colon X\to \R^2$ continuous function, $\boldsymbol{\p}^*_{\boldsymbol{a},\boldsymbol{b}}:=\min \{a, 1-a\}\max \{\frac{\p_1-b}{a}, \frac{\p_2+b}{1-a}\}$, with $\boldsymbol{a}:=(a,1-a)$, $\boldsymbol{b}:=(b,-b)$, and $a\in ]0,1[$, $b\in \R$.
\end{itemize}

Let $\p, \psi\colon X\to \R$ be continuous functions.
The $k$-bottleneck distance \cite{Cohen-Steiner2007} is a metric between persistence diagrams defined by
\[
d_{B, k}(\p, \psi):=d_B(\dgm_k(\p), \dgm_k(\psi))=\max _{\sigma} \max_{p\in \dgm_k(\p)} d(p, \sigma(p)),
\]
where $\sigma$ ranges over the multiset bijections between $\dgm_k(\p)$ and $\dgm_k(\psi)$.

Let
$\boldsymbol{\p},\boldsymbol{\psi}\colon X\to \R^2$ be continuous functions.
The $k$-th matching distance~\cite{CeDFFeFrLa13} is a metric between the persistence Betti number functions (rank invariants) defined by
\[
d_{\mathrm{match}, k}(\boldsymbol{\p}, \boldsymbol{\psi}):=\sup_{\boldsymbol{a}, \boldsymbol{b}} d_{B, k}(\boldsymbol{\p}^*_{\boldsymbol{a}, \boldsymbol{b}}, \boldsymbol{\psi}^*_{\boldsymbol{a}, \boldsymbol{b}}).
\]

\begin{remark}
\label{rem:asdkfjhaskjdfn}
Because of the definition of $d$, the $k$-bottleneck distance between persistence diagrams $\dgm_k(\p)$ and $\dgm_k(\psi)$ is of the form $c\lvert w_0- w_1\rvert$, where $c\in \{\frac{1}{2}, 1\}$, and $w_0, w_1$ are coordinates of the points of $\dgm_k(\p)\sqcup \dgm_k(\psi)$.
\end{remark}

\section{Convex matching distance}

Let us consider a continuous function $\bo \p=(\p_1,\p_2)\colon X\to\R^2$.
    For every $t\in[0,1]$, we set $\bo \p^t:=(1-t) \p_1+t \p_2$.
\begin{definition}
Considering $\bo \p,\bo \s \colon X \to \mathbb{R}^2$ and an integer $k$, we define the \textit{convex matching distance} $\cmd_k$ between $\bo \p$ and $\bo \s$ as
\begin{equation}\label{def:smd}
\cmd_k(\bo \p,\bo \s):= \max_{t\in[0,1]}d_{B,k} \left(\bo \p^t,\bo \s^t\right)=\max_{t\in[0,1]}d_{B}\left(\dgm_k(\bo \p^t),\dgm_k(\bo \s^t)\right).
\end{equation}
\end{definition}


\begin{remark}\label{rmk:well-defined}
We now show that (\ref{def:smd}) is well defined and that the maximum actually exists.
If the function $\boldsymbol{\varphi} = (\varphi_1,\varphi_2)$ is fixed, then for any $t_1,t_2 \in [0,1]$ we have
\begin{align*}
        \lVert \boldsymbol{\varphi}^{t_1} - \boldsymbol{\varphi}^{t_2} \rVert_\infty
        & = \big\lVert (1-t_1)\varphi_1 + t_1 \varphi_2 - \big((1-t_2)\varphi_1 + t_2 \varphi_2\big)\big\rVert_\infty \\
        & = \big\lVert \varphi_1 - t_1 \varphi_1 + t_1 \varphi_2 - \varphi_1 + t_2 \varphi_1 - t_2 \varphi_2 \big\rVert_\infty \\
        & = \lVert (t_1 - t_2)(\varphi_2 - \varphi_1) \rVert_\infty \\
        & = |t_1 - t_2| \,\lVert \varphi_1 - \varphi_2 \rVert_\infty .
\end{align*}
It follows that the map $t \mapsto \boldsymbol{\varphi}^t$ is $\lVert \varphi_1 - \varphi_2 \rVert_\infty$-Lipschitz (see also~\cite[Proposition 7.9]{BePe26}).
From the stability of the matching distance~\cite{CeDFFeFrLa13}—which implies
\(d_{B,k}(\boldsymbol{\varphi}^{t_1},\boldsymbol{\varphi}^{t_2})\le \lVert \boldsymbol{\varphi}^{t_1}-\boldsymbol{\varphi}^{t_2}\rVert_\infty\) and
\(d_{B,k}(\boldsymbol{\psi}^{t_1},\boldsymbol{\psi}^{t_2})\le \lVert \boldsymbol{\psi}^{t_1}-\boldsymbol{\psi}^{t_2}\rVert_\infty\)—and
by applying the
triangle inequality, we get
\begin{align*}
\Bigl|d_{B,k}\bigl(\boldsymbol{\varphi}^{t_1},\boldsymbol{\psi}^{t_1}\bigr)-
d_{B,k}\bigl(\boldsymbol{\varphi}^{t_2},\boldsymbol{\psi}^{t_2}\bigr)\Bigr|
\;& \le d_{B,k}\bigl(\boldsymbol{\varphi}^{t_1},\boldsymbol{\varphi}^{t_2}\bigr)+
d_{B,k}\bigl(\boldsymbol{\psi}^{t_1},\boldsymbol{\psi}^{t_2}\bigr)\\
 & \le  \lVert \boldsymbol{\p}^{t_1}-\boldsymbol{\p}^{t_2}\rVert_\infty + \lVert \boldsymbol{\psi}^{t_1}-\boldsymbol{\psi}^{t_2}\rVert_\infty\\
 & \le |t_1 - t_2| \,\lVert \varphi_1 - \varphi_2 \rVert_\infty + |t_1 - t_2| \,\lVert \psi_1 - \psi_2 \rVert_\infty.
\end{align*}
It follows that the map
\(t\mapsto d_{B,k}(\boldsymbol{\varphi}^{t},\boldsymbol{\psi}^{t})\) is continuous on \([0,1]\), as claimed.

We also observe that the map
\(\boldsymbol{\varphi}\mapsto\boldsymbol{\varphi}^t\) is \(1\)-Lipschitz for every \(t\in[0,1]\), since
\begin{align*}
        \lVert \boldsymbol{\varphi}^{t} - \boldsymbol{\psi}^{t} \rVert_\infty
        & = \big\lVert (1-t) \varphi_1+t \varphi_2 - \big((1-t) \psi_1+t \psi_2\big)\big\rVert_\infty \\
        & = \big\lVert (1-t) (\varphi_1-\psi_1)+t (\varphi_2 -\psi_2) \big\rVert_\infty \\
        & \le (1-t) \|\varphi_1-\psi_1\|_\infty+t \|\varphi_2 -\psi_2\|_\infty \\
        & \le (1-t) \|\boldsymbol{\varphi}-\boldsymbol{\psi}\|_\infty+t \|\boldsymbol{\varphi} -\boldsymbol{\psi}\|_\infty \\
        & = \|\boldsymbol{\varphi}-\boldsymbol{\psi}\|_\infty.
\end{align*}

\end{remark}

\begin{remark}\label{rmk:smoothness}
If $\bo\p$ is $k$-times differentiable, with $k$ possibly infinite, then so it is the function $\bo\p^t$ for each $t \in [0,1]$, due to being defined as the composition of $\bo\p=(\p_1,\p_2)$ with the smooth function $(s_1,s_2)\mapsto (1-t)s_1+ts_2$.
This is a key difference between the convex matching distance and the classical matching distance, for the operator sending $\bo\p$ to $\boldsymbol{\p}^*_{\boldsymbol{a},\boldsymbol{b}}$ preserves continuity only. 

Moreover, the function $t \mapsto \boldsymbol{\p}^t$ defined on the interval $[0,1]$ is smooth (i.e., infinitely differentiable) on $[0,1]$ (where for the endpoints we mean that the function has continuous left-sided or right-sided derivatives of all required orders). The claim has been already proven for the order 0 in Remark \ref{rmk:well-defined}.
For the other orders, it is enough to notice that, independently from the chosen $t$, the first derivative of the considered function is $-\p_1 + \p_2$  while, for higher orders, the derivative is
the null function.

\end{remark}
\begin{remark}\label{spieganome}
We call $\cmd_k$ the \emph{convex matching distance} because it is based on the convex combination of the components of $\bo \p$ and $\bo \s$.
This idea generalises to the case of functions with values in $\R^n$: given the function $\bo \p = (\p_1, \ldots, \p_n)$, we consider the map assigning to each vector $(t_1, \ldots, t_n) \in \R^n$ with $t_i \ge 0$ and $\sum_{i=1}^n t_i = 1$ the function $\bo \p^t := \sum_{i=1}^n t_i \p_i$.
\end{remark}

\begin{proposition}
    $\cmd_k$ is a pseudo-metric.
\end{proposition}
\begin{proof}
    It can be easily seen that $\cmd_k$ is a non-negative function, for which the symmetry property and the fact that, for any $\bo \p$, $\cmd_k(\bo \p, \bo \p) = 0$ are directly inherited from the bottleneck distance. Thus, to complete the proof, it remains to show that the triangle inequality holds.
    Since $d_{B,k}$ is a pseudo-metric, for any $\bo \p, \bo \s$ and $\boldsymbol{\xi}$ we have that
    \begin{align*}
        \cmd_k(\bo \p,\bo \s)&= \max_{t\in[0,1]}d_{B,k}(\bo \p^t,\bo \s^t)
        \\&\le \max_{t\in[0,1]}(d_{B,k}(\bo \p^t,\boldsymbol{\xi}^t)+ d_{B,k}(\boldsymbol{\xi}^t,\bo \s^t))
    \\&\le \max_{t\in[0,1]}d_{B,k}(\bo \p^t,\boldsymbol{\xi}^t)+ \max_{t\in[0,1]}d_{B,k}(\boldsymbol{\xi}^t,\bo \s^t)
    \\&= \cmd_k(\bo \p,\boldsymbol{\xi})+\cmd_k(\boldsymbol{\xi},\bo \s).
    \end{align*}
\end{proof}
\begin{remark}
We claim that in general $\cmd_k$ is not a metric. Since there exist two real-valued functions $f,g$ such that $f \ne g$ and $d_B(\dgm_k(f),\dgm_k(g))=0$, we can consider $\bo \p=(f,f)$ and $\bo \s=(g,g)$. Thus, $\bo \p^t=f$ and $\bo \s^t=g$ for any $t$. Then we obtain that $\cmd_k(\bo \p,\bo \s) =d_B(\dgm_k(f),\dgm_k(g))=0$, but $\bo \p \ne \bo \s$.
\end{remark}


\begin{proposition}[Stability of $\cmd_k$]\label{pro:stable}
For any $k$, $\cmd_k(\bo \p,\bo \s)\le\|\bo \p-\bo \s\|_\infty$.
\end{proposition}
\begin{proof}
For any $t\in [0,1]$, by the stability of the bottleneck distance~\cite{Cohen-Steiner2007} and Remark~\ref{rmk:well-defined},
\[
d_{B}\left(\dgm_k(\bo \p^t),\dgm_k(\bo \s^t)\right) \le\|\bo \p^t-\bo \s^t\|_\infty \le \|\bo \p-\bo \s\|_\infty.
\]
\end{proof}



\section{Comparison with the matching distance}

In this section, we compare the newly introduced convex matching distance with the classical matching distance \cite{CeDFFeFrLa13}.

In Example \ref{exKC1}, we exhibit two functions for which the convex matching distance is strictly greater than the classical matching distance, thereby showing that, at least in some cases, the new distance is more discriminative in distinguishing data.

Vice versa, Example \ref{exKC2} provides functions for which the convex matching distance is strictly smaller than the classical matching distance.

\begin{example}\label{exKC1}
Let us consider the circumference $S$ in $\R^3$ of
radius $\sqrt{2}$ and center $(0,0,0)$, whose points belong to the plane of equation $x+z=0$. Then, let us consider the circle $C$ bounded by $S$ and the conical set $K$ obtained by taking the union of the segments joining the point $(1,0,1)$ to each point of $S$ (see Figure~\ref{cono}).
\begin{figure}
\centering
\includegraphics[width=0.5\linewidth]{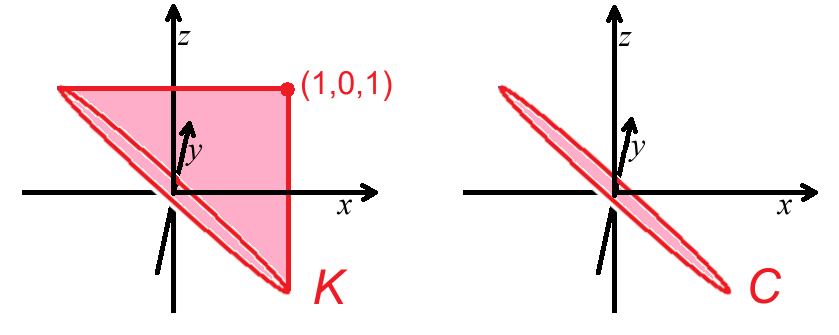}
\caption{The sets $K$ and $C$ cited in the Examples \ref{exKC1}, \ref{exKC2}.}
\label{cono}
\end{figure}
We set $\boldsymbol{\p}(x,y,z):=(x,z)$ for $(x,y,z)\in K$ and $\boldsymbol{\psi}(x,y,z):=(x,z)$ for $(x,y,z)\in C$.
Given that $K$ and $C$ are homeomorphic to the closed hemisphere $\Sigma:=\left\{(x,y,z)\in\R^3: x^2+y^2+z^2=1, x+z\ge 0\right\}$,
we can find two homeomorphisms $\bo h_K:\Sigma\to K$, $\bo h_C:\Sigma\to C$. Therefore, we can consider the continuous functions
$\hat{\bo \p}:\Sigma\to\R^2$, $\hat{\bo \p}:=\bo \p\circ \bo h_K$, $\hat{\bo \psi}:\Sigma\to\R^2$, $\hat{\bo \psi}:=\bo \psi\circ \bo h_C$.
Let us now compute the following persistence diagrams in degree $1$:
$$
\dgm_1\left(\hat{\boldsymbol{\p}}^*_{\boldsymbol{a},\boldsymbol{b}}\right)=\dgm_1\left({\boldsymbol{\p}}^*_{\boldsymbol{a},\boldsymbol{b}}\right)=\left\{\Delta\right\}=\dgm_1\left({\boldsymbol{\psi}}^*_{\boldsymbol{a},\boldsymbol{b}}\right)=\dgm_1\left(\hat{\boldsymbol{\psi}}^*_{\boldsymbol{a},\boldsymbol{b}}\right)\mathrm{\ for\ any\ }{\boldsymbol{a},\boldsymbol{b}};$$
$$ \dgm_1\left(\hat{\boldsymbol{\p}}^{\frac{1}{2}}\right)=\dgm_1\left({\boldsymbol{\p}}^{\frac{1}{2}}\right)=\left\{\Delta, (0,1)\right\} \mathrm{\ and\ }\dgm_1\left(\hat{\boldsymbol{\psi}}^{\frac{1}{2}}\right)=\dgm_1\left({\boldsymbol{\psi}}^{\frac{1}{2}}\right)=\left\{\Delta\right\}.$$
    Therefore, $d_{\mathrm{match},1}(\hat{\boldsymbol{\p}},\hat{\boldsymbol{\psi}})=0$ and $d_B\left(\hat{\boldsymbol{\p}}^{\frac{1}{2}},\hat{\boldsymbol{\psi}}^{\frac{1}{2}}\right)=\frac{1}{2}$. Hence $\cmd_1(\hat{\boldsymbol{\p}},\hat{\boldsymbol{\psi}})\ge \frac{1}{2}>0=d_{\mathrm{match},1}(\hat{\boldsymbol{\p}},\hat{\boldsymbol{\psi}})$.
\end{example}

However, the convex matching distance is not always greater than the classical matching distance, as shown by the following example.

\begin{example}\label{exKC2}
If we consider degree $0$ instead of degree $1$, the functions given in Example \ref{exKC1} show a case in which the convex matching distance is strictly smaller than the classical matching distance.
In fact, it is easy to verify that

    $$\dgm_0\left(\hat{\boldsymbol{\p}}^*_{(\frac{1}{2},\frac{1}{2}),(0,0)}\right)=
    \dgm_0\left({\boldsymbol{\p}}^*_{(\frac{1}{2},\frac{1}{2}),(0,0)}\right)=\left\{\Delta,(0,1),(0,\infty)\right\};$$
    $$\dgm_0\left(\hat{\boldsymbol{\psi}}^*_{(\frac{1}{2},\frac{1}{2}),(0,0)}\right)=
    \dgm_0\left({\boldsymbol{\psi}}^*_{(\frac{1}{2},\frac{1}{2}),(0,0)}\right)=\left\{\Delta,(0,\infty)\right\};$$
    $$ \dgm_0\left(\hat{\boldsymbol{\p}}^{t}\right)=\dgm_0\left({\boldsymbol{\p}}^{t}\right)=\left\{\Delta, (1-2t,\infty)\right\} =\dgm_0\left({\boldsymbol{\psi}}^{t}\right)=\dgm_0\left(\hat{\boldsymbol{\psi}}^{t}\right) \mathrm{\ for\ any\ } t\in [0,1].$$
    Therefore, $d_{\mathrm{match},0}(\hat{\boldsymbol{\p}},\hat{\boldsymbol{\psi}})\ge\frac{1}{2}$ and $d_B\left(\dgm_0({\boldsymbol{\p}}^{t}),\dgm_0({\boldsymbol{\psi}}^{t})\right)=0$ for any $t\in[0,1]$.
    Hence, $\cmd_0(\hat{\boldsymbol{\p}},\hat{\boldsymbol{\psi}})=0<\frac{1}{2}\le d_{\mathrm{match},0}(\hat{\boldsymbol{\p}},\hat{\boldsymbol{\psi}})$.
\end{example}

\begin{remark}\label{clarifyingrem}
The above examples can be suitably modified to satisfy more restrictive conditions on the smoothness of the involved functions or on their domains, taken to be manifolds.
This shows that their existence is not due to a lack of regularity properties but rather follows from the very definition of the two distances.
By identifying the boundaries of two disjoint copies of $K$ and $C$, respectively, the above examples can be adapted,
thus obtaining two smooth functions
$\tilde{\bo \p},\tilde{\bo \psi}\colon\mathbb S^2\to\R^2$
such that:
\begin{itemize}
    \item $d_{\mathrm{match},1}(\tilde{\boldsymbol{\p}},\tilde{\boldsymbol{\psi}})=0<\frac{1}{2}\le\cmd_1(\tilde{\boldsymbol{\p}},\tilde{\boldsymbol{\psi}})$;
    \item $\cmd_0(\tilde{\boldsymbol{\p}},\tilde{\boldsymbol{\psi}})=0<\frac{1}{2}\le d_{\mathrm{match},0}(\tilde{\boldsymbol{\p}},\tilde{\boldsymbol{\psi}})$.
\end{itemize}
\end{remark}

\section{Pareto grid and Position Theorem}\label{sec:position}


Let $M$ be a closed (i.e., compact and without boundary) smooth Riemannian manifold of dimension $r\ge2$ and $\bo\p=(\p_1,\p_2)\colon M\to \mathbb{R}^2$ a smooth function.

\begin{definition}\label{def:prop_functions}
\begin{enumerate}
    \item The \emph{Jacobi set}, $\mathbb{J}(\bo{\p})$, is the set of all points $x\in M$ at which the gradients of $\p_1$ and $\p_2$ are linearly dependent, namely
    $x\in\mathbb{J}(\bo{\p})$ if and only if there exists $(\lambda_1, \lambda_2)\neq (0,0)$ such that $\lambda_1\nabla\p_1(x) + \lambda_2\nabla\p_2(x) = 0$.
    \item The \emph{Pareto critical set} of $\bo{\p}$, $\mathbb{J}_P(\bo{\p})$, is the set of all points $x \in \mathbb{J}(\bo{\p})$ such that there exists
    $(\lambda_1, \lambda_2)\neq (0,0)$ with $\lambda_1,\lambda_2\geq 0$ such that $\lambda_1\nabla\p_1(x) + \lambda_2\nabla\p_2(x) = 0$.
    The elements of $\mathbb{J}_P(\bo{\p})$ are called \emph{Pareto critical points}.
    Note that $\mathbb{J}_P(\bo{\p})$ contains both the critical points of $\p_1$ and those of $\p_2$.
\end{enumerate}
\end{definition}

Unless otherwise specified, throughout the rest of the article the smooth functions considered will satisfy the following conditions.
\begin{ass}\label{ass:pareto}
\
\begin{enumerate}
    \item No point $x$ exists in $M$ at which both $\nabla \p_1(x)$ and $\nabla \p_2(x)$ vanish.
    \item $\mathbb{J}(\bo\p)$ is a 1-dimensional manifold smoothly embedded in $M$, consisting of finitely many components, each diffeomorphic to a circle.
    \item $\mathbb{J}_P(\bo\p)$ is a 1-dimensional closed submanifold of $\mathbb{J}(\bo\p)$ with boundary.
    \item
    The connected components of $\mathbb{J}_P(\bo\p)\setminus\mathbb{J}_C(\bo\p)$ are finite in number, each being diffeomorphic to an interval,
    where
    $\mathbb{J}_C(\bo\p)$ is
    the set of cusp points of $\bo \p$, i.e.,
    the set of points of $\mathbb{J}(\bo\p)$ at which the restriction of $\bo\p$ to $\mathbb{J}(\bo\p)$ fails to be an immersion. 
    In other words, $\mathbb{J}_C(\bo\p)$ is the subset of $\mathbb{J}(\bo\p)$ where both $\nabla \p_1$ and $\nabla \p_2$ are orthogonal to $\mathbb{J}(\bo\p)$.
    With respect to any parameterisation of each component, one of $\p_1$ and $\p_2$ is strictly increasing and the other is strictly decreasing. Each component may meet critical points of $\p_1$ or $\p_2$ only at its endpoints.
\end{enumerate}
\end{ass}


We call the closures of the images of the connected components of $\mathbb{J}_P(\bo\p) \setminus \mathbb{J}_C(\bo\p)$ (proper) \emph{contours} \cite{coher_match}, and denote by $\mathrm{Ctr}(\bo\p)$ the set of contours of $\p$.
Given two smooth functions $\bo\p, \bo \s \colon M\to \mathbb{R}^2$, we set
$\mathrm{Ctr}(\bo \p,\bo \s):=\mathrm{Ctr}(\bo\p)\cup \mathrm{Ctr}(\bo\s)$.
It is important to observe that Assumption \ref{ass:pareto} is generic within the class of smooth functions from $M$ to $\R^2$ \cite{wan}.

In our setting, we make one further technical assumption.
\begin{ass}\label{ass:nocusps}
\
\begin{enumerate}
    \item The set $\mathbb J_C(\bo\p)$ of cusp points is empty.
    \item Each contour $\bo\alpha \in\mathrm{Ctr}(\bo\p)$ is a smooth regular curve in $\mathcal C^\infty([0,1];\R^2)$.
\end{enumerate}
\end{ass}




\begin{definition}
The \emph{Pareto grid} of $\bo \p$ is the subset of $\R^2$ defined as the image $\bo\p(\mathbb{J}_P(\bo\p))$.
\end{definition}

Figure~\ref{figura:data} shows the Pareto grid of the projection of a 2-sphere in $\R^2$.

\begin{figure}[h]
\begin{center}
\includegraphics[width=0.3\textwidth]{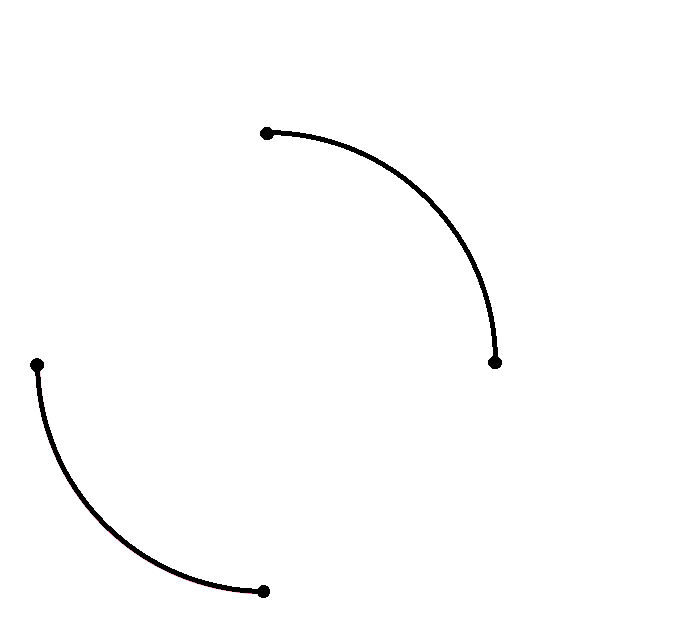}
\end{center}
\caption{The Pareto grid of the function $\bo \p:\mathbb S^2\to\R^2$ defined by setting $\bo \p(x,y,z):=(x,z)$.} 
\label{figura:data}
\end{figure}


Before presenting the main result of this section, we recall the following well-established theorem, which provides a useful relation between the persistence diagram of a smooth function and its critical values.
This statement is a technical but standard fact in the literature, and we include it here without proof
(cf., e.g., Theorem 3.1 in \cite{Simms_1964}, and \cite{Fr96}).

\begin{theorem}\label{thm_coord_cnpts_crit_values}
Let $M$ be a closed smooth manifold and $\p\colon M\to \R$ a smooth function.
If $w$ is a finite coordinate of a point ${p}\in\dgm_k(\p)\setminus \{\Delta\}$, then $w$ is a critical value for $\p$.
\end{theorem}
We are now ready to state the main result of this section.
Relying on
\cref{thm_coord_cnpts_crit_values},
it provides a geometric characterisation of the points in the persistence diagram of $\bo \p^t$ in terms of the
contours associated with $\bo \p$.
Next theorem is analogous to~\cite[Theorem 2]{coher_match}.
In what follows, the symbol $\cdot$ denotes the dot product.


\begin{theorem}[Position Theorem for $\bo \p^t$]
\label{thm:position}
Let $t \in [0,1]$ and let
$w$ be a finite coordinate of a point in $\dgm_k(\bo \p^t)\setminus \{\Delta\}$.
Then there exist a
contour
denoted as $\boldsymbol{\alpha} = (\alpha_1, \alpha_2) \colon [0,1] \to \mathbb{R}^2$ and a $\bar{\tau} \in [0,1]$ such that:
\begin{enumerate}
\item $\boldsymbol{\alpha}(\bar{\tau}) \cdot (1 - t, t) = w$;
\item $\dfrac{d\boldsymbol{\alpha}}{d\tau}(\bar{\tau}) \cdot (1 - t, t) = 0.$
\end{enumerate}
\end{theorem}


\begin{proof}
By Theorem~\ref{thm_coord_cnpts_crit_values} we know that $w$ must be a critical value for $\bo \p^t$.
Therefore, there exists a point $\bar x \in M$ such that $w = \bo \p^t(\bar x)$ and $\nabla \bo \p^t(\bar x) = \boldsymbol{0}$, where $\boldsymbol{0}$ denotes the null vector.
Hence, $(1 - t)\nabla \p_1(\bar x) + t\nabla \p_2(\bar x) = \boldsymbol{0}$.
This means that $\bar x$ is a Pareto critical point of $\bo \p$.
Therefore, there exist a
contour
$\boldsymbol{\alpha} = (\alpha_1, \alpha_2) : [0,1] \to \mathbb{R}^2$ and a $\bar \tau \in [0,1]$ such that $\bo \p(\bar x) = \boldsymbol{\alpha}(\bar \tau)$, and
\[
w=\bo \p^t(\bar x)=(1-t) \p_1(\bar x)+t \p_2(\bar x)=\bo \p(\bar x)\cdot(1-t,t)=\boldsymbol{\alpha}(\bar \tau)\cdot(1-t,t).
\]
Consider a smooth regular curve 
$\boldsymbol{\gamma}:
[0,1]
\to M$ such that $\boldsymbol{\alpha}(\tau)=\bo \p\circ \boldsymbol{\gamma}(\tau)$ and $\boldsymbol{\gamma}(\bar \tau)=\bar x$,
which implies that $\boldsymbol{\alpha}(\bar \tau)=(\alpha_1(\bar \tau),\alpha_2(\bar \tau))=(\p_1(\boldsymbol{\gamma}(\bar \tau)),\p_2(\boldsymbol{\gamma}(\bar \tau))=(\p_1(\bar x),\p_2(\bar x))$.
By Assumption \ref{ass:nocusps}, $\mathbb J_C(\bo\p)$ is empty, hence $\frac{d}{d\tau}(\bo\p \circ \bo\gamma(\tau))$ does not vanish for any $\tau \in [0,1]$.
Therefore, we have that
\begin{align}
\begin{split}
\label{eq:proof_position}
        0&=\boldsymbol{0}\cdot \frac{d\boldsymbol{\gamma}}{d\tau}(\bar \tau)\\
        &=\nabla\bo \p^t(\bar p)\cdot\frac{d\boldsymbol{\gamma}}{d\tau}(\bar \tau)\\
        &=\left((1-t)\nabla \p_1(\bar x)+t\nabla \p_2(\bar x))\right)\cdot\frac{d\boldsymbol{\gamma}}{d\tau}(\bar \tau)\\
        &=(1-t)\nabla \p_1(\bar x)\cdot\frac{d\boldsymbol{\gamma}}{d\tau}(\bar \tau)+t\nabla \p_2(\bar x)\cdot\frac{d\boldsymbol{\gamma}}{d\tau}(\bar \tau)\\
        &=(1-t)\frac{d\p_1\circ\boldsymbol{\gamma}}{d\tau}(\bar \tau)+t\frac{d\p_2\circ\boldsymbol{\gamma}}{d\tau}(\bar \tau)\\
        &=(1-t)\frac{d\alpha_1}{d\tau}(\bar \tau)+t\frac{d\alpha_2}{d\tau}(\bar \tau)\\
        &=\frac{d\boldsymbol{\alpha}}{d\tau}(\bar \tau)\cdot (1-t,t).
\end{split}
\end{align}

\end{proof}

It may be useful here to observe how the Position Theorem \ref{thm:position} can be used to find the finite coordinates of the points in
$\dgm_k({\bo\p}^t)\setminus \{\Delta\}$.
The idea is to consider the family of parallel lines having the direction of the vector $(1-t,t)$ and to look for the points where these lines intersect the contours associated with the function
$\bo\p\colon M\to \R^2$ orthogonally.
For each such point $P$, one computes the scalar product $P\cdot (1-t,t)$.
As $P$ varies, one thus obtains all the finite coordinates of the points in
$\dgm_k({\bo\p}^t)\setminus \{\Delta\}$.

\section{Special values}

For the following definitions we need to introduce the notion of signed radius:

\begin{definition}
\label{def:signed_radius}
Let $\bo\alpha$ be a regular smooth curve in $\R^2$ and let $P \in \bo\alpha$.
If the curvature of $\bo\alpha$ at $P$ is non-zero,
    let $(x,y)$ and $\rho$ denote the centre and radius of the osculating circle to $\bo\alpha$ at $P$,
    and define the \emph{signed radius} of $\bo\alpha$ at $P$ as
    $\ell=\mathrm{sign}((P - (x,y))\cdot (1,0))\rho$.
\end{definition}

\begin{definition}
\label{def:special}
The \emph{special set}
of $(\bo \p,\bo \s)$, denoted by $\mathrm{Sp}(\bo \p,\bo \s)$, consists of the values $t\in [0,1]$ such that at least one of the following hold:
\begin{enumerate}
\item There exists a line $r$ with direction $(1-t,t)$ intersecting orthogonally a contour $\bo\alpha \in \mathrm{Ctr}(\bo\p,\bo\s)$ at one of its endpoints.
\item There exist lines $r_1,r_2,r_3,r_4$ with direction $(1-t,t)$ and contours $\bo\alpha_1,\bo\alpha_2,\bo\alpha_3,\bo\alpha_4\in\mathrm{Ctr}(\bo \p,\bo \s)$
such that $r_i$ intersects $\bo\alpha_i$ orthogonally at $P_i$, for each $i \in \left\{1,2,3,4\right\}$, with
$\{P_1,P_2\}\neq\{P_3,P_4\}$,
and $(P_1-P_2)\cdot(1-t,t) = c ((P_3-P_4)\cdot(1-t,t))$, for $c \in \{\frac{1}{2},1\}$. 
\item There exist lines $r_1,r_2$ with direction $(1-t,t)$ and contours $\bo\alpha_1,\bo\alpha_2 \in \mathrm{Ctr}(\bo \p,\bo \s)$
such that $r_i$ intersects the interior of $\bo\alpha_i$ orthogonally at $P_i$, for $i$ in $\left\{1,2\right\}$, with $P_1\neq P_2$ and
one of the following conditions holds:
\begin{itemize}
    \item $\bo\alpha_1$ has null curvature at $P_1$,
    \item $\bo\alpha_2$ has null curvature at $P_2$,
    \item $\ell_1$ and $\ell_2$ are defined and $\ell_1=\ell_2$,
    \item $\ell_1$ and $\ell_2$ are defined, with $\ell_1\ne \ell_2$, and
    \[t = \frac{\zeta}{1+\zeta} \text{ with } \zeta := \tan\left(\frac{1}{2}\arcsin\left(1-\left(\frac{(y_1-y_2)-(x_1-x_2)}{
    (\ell_1-\ell_2)
    } \right)^2\right)\right),\]
\end{itemize}
where $\ell_i$ is the signed radius and $(x_i, y_i)$ is the center of the osculating circle of $\bo \alpha_i$ at $P_i$, for $i$ in $\{1,2\}$.
\end{enumerate}
\end{definition}

\begin{remark}
\label{rem:0e1}
    Note that $t=0$ and $t=1$ are always special values for each $\bo \p,\bo\s$.
    For example, to see that $t=0$ is special, let $x$ in $M$ be a point at which $\p_1$ attains its maximum.
    Then
    $\bo\p(x)$ is the endpoint of some contour $\bo \alpha\in \text{Ctr}(\bo\p)$.
    In particular, let $\gamma \colon [0,1] \to M$ be such that $\bo\alpha(\tau) = \bo \p \circ \gamma (\tau)$ and suppose that $\gamma(0) = x$.
    Therefore,
    \[0 = \nabla \p_1(x) \cdot \frac{d\gamma}{d\tau}(0) = \frac{d(\p_1 \circ \gamma)}{d\tau}(0) = \frac{d \alpha_1}{d\tau}(0)= \frac{d\bo\alpha}{d\tau}(0)\cdot (1,0).\]

    Hence $(1,0)$ is orthogonal to $\bo\alpha$ at an endpoint.
    One shows in an analogous way that $t=1$ is also a special value.
\end{remark}

\begin{remark}
\label{rmk:intuition_special}
    The intuition behind Definition \ref{def:special}.2 is the following: if $t\in[0,1]$ satisfies Definition \ref{def:special}.2, then $t$ is a value at which the optimal matching between $\dgm(\bo\p^t)$ and $\dgm(\bo\s^t)$ may change abruptly.
    In particular, if $\sigma$ is such an optimal matching and
    $p_1,p_2,p_3,p_4 \in \dgm(\bo\p^t)\cup\dgm(\bo\s^t)$
    are such that $\cost \sigma = w_1 - w_2 = c(w_3-w_4)$, where $w_i$ is a finite coordinate of $p_i$, $i=1,\ldots,4$, then, in a neighbourhood of $t$, the optimal matching could be either the one matching $p_1$ to $p_2$ or the one matching $p_3$ to $p_4$, but not the other.
    A similar motivation led to the definition of special set in \cite{EFQT2023},
    relatively to the classical matching distance.
\end{remark}


\begin{remark}
    	\begin{figure}[htbp]
		\centering

        \resizebox{0.75\textwidth}{!}{%
        \begin{minipage}{0.4\textwidth}
			\centering
			\begin{tikzpicture}
				\draw[thick,samples=200,domain=0:2.5] plot(\x,{5-0.5*(\x)*(\x)});
				\node at (2.25,3.25) {$\bo\alpha$};
				\fill (1,4.5) circle(2pt);
				\fill (-1,2.5) circle(2pt) node[below] {$(x,y)$};
				\draw[dashed] (-1,2.5) circle (2.8284);
				\node at (-3,3.85) {$\hat{\bo\alpha}$};
				\draw[blue,thick] (1,4.5) -- (-1,2.5) node[midway,below] {$\rho$};
			\end{tikzpicture}
		\end{minipage}%
        \hspace{0.2\textwidth}
		\begin{minipage}{0.4\textwidth}
			\centering
			\begin{tikzpicture}
				\draw[thick,samples=200,domain=-4.5:-2] plot(\x,{0.5*(\x+2)*(\x+2)});
				\node at (-4.25,1.75) {$\bo\alpha'$};
				\fill (-3,0.5) circle(2pt);
				\fill (-1,2.5) circle(2pt);
				\draw[dashed] (-1,2.5) circle (2.8284) node[above] {$(x',y')$};
				\node at (-3,3.85) {$\hat{\bo\alpha'}$};
				\draw[blue,thick] (-3,0.5) -- (-1,2.5) node[midway,below] {$\rho'$};
			\end{tikzpicture}
		\end{minipage}
		}
		\caption{Osculating circles $\hat{\bo\alpha}$ and $\hat{\bo\alpha'}$ of centres $(x,y),(x',y')$ and radii $\rho,\rho'$ to a negative and positive parabola denoted by $\bo\alpha$ and $\bo\alpha'$, respectively, at a non-critical point.
        The signed radius $\ell$ from Definition \ref{def:signed_radius} is positive on the leftmost case and negative on the rightmost one.}
        \label{fig:osculating_circles}
	\end{figure}
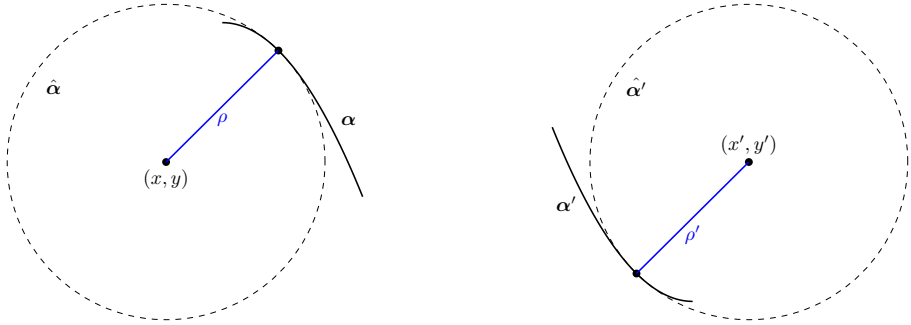
    We note that the radius $\ell^i$ appearing in Definition \ref{def:special}.3 has positive sign whenever $P_i$ is above and to the right of the centre of the osculating circle, and negative when it is below and on the left.
    These are the only cases possible due to Assumption \ref{ass:pareto}.4, which are illustrated in Figure \ref{fig:osculating_circles}.

\end{remark}

The next result characterises the values of $t$ at which the convex matching distance is attained.
These values always exist by Remark \ref{rmk:well-defined}.
We observe that when the convex matching distance is zero, every $t \in [0,1]$ realises it.

\begin{theorem}
\label{thm:main}
If $\cmd_k(\bo \p,\bo \s)>0$ and it is realised at $\bar t$, that is,
\[
\cmd_k(\bo \p,\bo \s)= d_{B,k}\left(\bo \p^{\bar t},\bo \s^{\bar t}\right),
\]
then $\bar t \in \mathrm{Sp}(\bo \p, \bo \s)$.
\end{theorem}


\begin{proof}
By contradiction, suppose $\bar t \not\in \Sp(\bo \p,\bo \s)$ and $\cmd_k(\bo \p,\bo \s)=d_{B,k}\left(\bo \p^{\bar t},\bo \s^{\bar t}\right)$.

By Remark \ref{rem:asdkfjhaskjdfn} there exist coordinates $w_0,w_1\in \R$ of points in $\dgm_k(\bo\p^{\bar t}) \cup\dgm_k(\bo\s^{\bar t}) $
and a constant $c\in \{\frac{1}{2},1\}$ such that $\cmd_k(\bo \p,\bo \s) = c\vert w_0-w_1\vert > 0$.
Moreover, by Theorem \ref{thm:position}, there exist two
contours $\bo \alpha, \bo \beta \in \mathrm{Ctr}(\bo \p,\bo \s)$ such that, for some $\tau_\alpha$ and $\tau_\beta$ in $[0,1]$:
\begin{align*}
    \bo\alpha(\tau_\alpha) \cdot (1-\bar t,\bar t) = w_0,\\
    \bo\beta(\tau_\beta) \cdot (1-\bar t,\bar t) = w_1,
\end{align*}
where with a slight abuse of notation we denote with $\bo \alpha(\tau), \bo \beta(\tau')$ some smooth parametrisations of $\bo \alpha$ and $\bo \beta$ such that $(1-\bar t,\bar t)$ is orthogonal to $\bo \alpha$ and $\bo \beta$ at $\bo\alpha(\tau_\alpha)$ and $\bo\beta(\tau_\beta)$, respectively.
By Definition \ref{def:special}.2 no two distinct pairs of coordinates can realise $d_{B, k}\left(\bo \p^{\bar t},\bo \s^{\bar t}\right)$,
hence $\{w_0,w_1\}$
is unique.
Moreover, note that $\tau_\alpha$ and $\tau_\beta$ cannot be in $\{0,1\}$, since $0$ and $1$ are special values (see Remark \ref{rem:0e1}),
it is enough to consider
the interior of the contours $\bo \alpha$ and $\bo \beta$.

Now, we locally reparametrise the interior of the contours $\bo \alpha$ and $\bo \beta$ in neighborhoods of $\bo\alpha(\tau_\alpha)$ and $\bo\beta(\tau_\beta)$, respectively.
Let $\theta(t):= \arctan\left(\frac{t}{1-t}\right) \in ]0,\frac{\pi}{2}[$ for $t \in ]0,1[$
and let $\bar \theta := \theta(\bar t)$.
Consider the osculating circles $\hat {\bo \alpha}$ and $\hat {\bo \beta}$ at $\bo\alpha(\tau_\alpha)$ and $\bo\beta(\tau_\beta)$,
which may be parametrised as $\hat{\bo\alpha}(\theta) := (x_{\bo\alpha},y_{\bo\alpha}) + \ell_{\bo\alpha}(\cos \theta,\sin \theta)$ and $\hat{\bo\beta}(\theta) := (x_{\bo\beta},y_{\bo\beta}) + \ell_{\bo\beta}(\cos \theta,\sin \theta)$, where $\ell_{\bo\alpha},\ell_{\bo\beta}$ and $(x_{\bo\alpha},y_{\bo\alpha}),(x_{\bo\beta},y_{\bo\beta})$ denote the
signed radii -- in the sense of Definition \ref{def:signed_radius} -- and centres of the osculating circles, and the parameter $\theta$ varies in $[0,2\pi[$.
Note that, $\hat{\bo\alpha}(\bar \theta) = \bo\alpha(\tau_\alpha)$.
Moreover,
by Definition \ref{def:special}.3 and Assumption \ref{ass:nocusps},
we have $\vert \ell_{\bo\alpha}\vert , \vert \ell_{\bo\beta}\vert  \in ]0,\infty[$.
Consider the line $r(\theta) := \{(x_{\bo\alpha},y_{\bo\alpha}) + s(\cos \theta,\sin \theta) \mid s \in \R\}$, and let $\bo\alpha(\theta)$ be the local and smooth parametrisation of $\bo \alpha$ given by the intersection $\bo \alpha \cap r(\theta)$
in a neighbourhood of $\bar \theta$.
The existence of this parametrisation is by the transversality between $\bo\alpha$ and $r(\theta)$ provided by Assumption \ref{ass:pareto}.

By reparametrising $\bo\beta$ in a neighborhood of $\bo \beta(\tau_\beta)$ in the same fashion, we obtain $\bo \beta(\bar \theta) = \bo\beta(\tau_\beta)$.
Thus, we can consider the difference $w_0(\theta) - w_1(\theta) = \bo \alpha(\theta) \cdot (1-t,t) - \bo \beta(\theta) \cdot (1-t,t)$.
To finish the proof we will show that $\dot w_0(\bar \theta) - \dot w_1(\bar \theta) \ne 0$, where we denote with the upper dot the derivative with respect to $\theta$.

By definition of $\theta$ and osculating circles, we have $\bo\alpha(\bar\theta) = \hat{\bo\alpha}(\bar\theta)$, $\bo\beta(\bar\theta) = \hat{\bo\beta}(\bar\theta)$, $\dot{\bo\alpha}(\bar\theta) = \dot{\hat{\bo\alpha}}(\bar\theta)$ and $\dot{\bo\beta}(\bar\theta) = \dot{\hat{\bo\beta}}(\bar\theta)$.
Noting that $(1-t,t) = \left(\frac{\cos \theta}{\cos \theta +\sin\theta},\frac{\sin \theta}{\cos \theta +\sin\theta}\right)$, we can explicitly compute the derivative $\dot w_0(\theta)$ in a neighborhood of $\bar\theta$:
\begin{align}
\begin{split}
    \label{eq:dotw0}
    \dot w_0(\theta) &= \frac{\partial}{\partial \theta} \left[\bo\alpha(\theta) \cdot  \left(\frac{\cos \theta}{\cos \theta +\sin\theta},\frac{\sin \theta}{\cos \theta +\sin\theta}\right)\right] \\
    &= \dot{\bo\alpha}(\theta)\cdot \left(\frac{\cos \theta}{\cos \theta +\sin\theta},\frac{\sin \theta}{\cos \theta +\sin\theta}\right)  +\bo\alpha(\theta)\cdot \frac{(-1,1)}{(\cos \theta + \sin \theta)^2}.
    \end{split}
\end{align}

The denominators in the previous expression do not vanish since $\theta \in ]0,\frac{\pi}{2}[$.
As $\dot{\bo\alpha}(\bar \theta) = \dot{\hat{\bo\alpha}}(\bar \theta)$,
we can rewrite \eqref{eq:dotw0} using the parametrisation of the osculating circle $\hat{\bo\alpha}$:
\begin{align}
\begin{split}
    \label{eq:dotw0bartheta}
    \dot w_0(\bar\theta) &= \left[\ell_{\bo\alpha}(-\sin\theta,\cos \theta)\cdot \left(\frac{\cos \theta}{\cos \theta +\sin\theta},\frac{\sin \theta}{\cos \theta +\sin\theta}\right)\right. \\
    &\left.\quad +\left((x_{\bo\alpha},y_{\bo\alpha}) + \ell_{\bo\alpha}(\cos \theta,\sin \theta)\right)\cdot \frac{(-1,1)}{(\cos \theta + \sin \theta)^2}\right]_{\theta = \bar\theta} \\
    &=\left[\frac{y_{\bo\alpha}-x_{\bo\alpha}+\ell_{\bo\alpha}(\sin\theta - \cos\theta)}{(\cos \theta + \sin \theta)^2}\right]_{\theta = \bar\theta}.
\end{split}
\end{align}


By using the osculating circle $\hat{\bo\beta}$ we can replicate the same reasoning to compute $\dot w_1(\bar\theta)$. Substracting both formulas we obtain:
\begin{equation}
\label{eq:aiuto}
    \frac{d}{d\theta}\left[w_0(\theta)-w_1(\theta)\right]_{\theta = \bar\theta} = 0 \quad \text{if and only if} \quad \cos \bar \theta - \sin \bar \theta = \frac{(y_{\bo\alpha}-y_{\bo\beta})-(x_{\bo\alpha}-x_{\bo\beta})}
    {\ell_{\bo\alpha}-\ell_{\bo\beta}},
\end{equation}
where the radii $\ell_{\bo\alpha},\ell_{\bo\beta}$ have positive or negative sign according to Definition \ref{def:signed_radius}.


Note that the denominator on the right-hand side of \eqref{eq:aiuto} does not vanish by Definition \ref{def:special}.3. 
Recall the trigonometric formula $\left(\cos \theta - \sin \theta\right)^2 = 1 - \sin 2\theta$, which applied to \eqref{eq:aiuto} yields
\[\bar \theta = \frac{1}{2}\arcsin\left(1-\left(\frac{(y_{\bo\alpha}-y_{\bo\beta})-(x_{\bo\alpha}-x_{\bo\beta})}{
\rho_{\bo\alpha}-\rho_{\bo\beta}
}\right)^2\right).\]

Lastly, applying back the change of variables $\tan \theta = \frac{t}{1-t}$ to the left-hand side of the last equation with $t=\bar t$
yields exactly the condition
in Definition \ref{def:special}.3, implying $\bar t \in \mathrm{Sp}(\bo\p,\bo\s).$
This generates the desired contradiction.


\end{proof}

\begin{remark}\label{remarksuint}
In the proof of Theorem \ref{thm:main}, we use the fact that the parametrisation of $\bo \alpha$ defined by the intersection $\bo \alpha \cap r(\theta)$ is smooth in a neighborhood of $\bo\alpha(\tau_\alpha)$.
The fact that $\bo \alpha(\theta)$ is a well-defined and smooth curve follows from the assumption that all lines passing through the point $(x_{\bo\alpha}, y_{\bo\alpha})$ intersect the contour $\bo \alpha$ transversally.
\end{remark}

\section{Experiments}
\label{sec:experiments}
Throughout the section, every function $\boldsymbol{\p}=(\p_1,\p_2)\colon X\to \R^2$ takes values in $[0, 1]^2$. This prevents scale-dependent artifacts in both the convex combinations and the bottleneck distance computations. Moreover, each point cloud will lie in the unit cube. We remind that $\boldsymbol{a} = (a, 1-a)$ and $\boldsymbol{b} = (b, -b)$.

The aim of this section is to evaluate the convex matching distance against the classical matching distance. In practice, both distances are approximated by evaluating the maximum over finite grids. Specifically, we evaluate three scenarios:
\begin{itemize}
    \item \textbf{CMD:} approximation of $\cmd_k$ by restricting to a uniform grid of $11$ values $t \in \{0, 0.1, \dots, 1\}$ in the convex combination $(1-t)\varphi_1 + t\varphi_2$;
    \item \textbf{MD~121:} approximation of $d_{\mathrm{match},k}$ by restricting to a fine $11 \times 11$ grid with $a$ ranging from $0.05$ to $0.95$ and $b \in \{0, 0.1, \dots, 1\}$ in the parameter space for $\boldsymbol{\p}^*_{\boldsymbol{a},\boldsymbol{b}}, \boldsymbol{\psi}^*_{\boldsymbol{a},\boldsymbol{b}}$;
    \item \textbf{MD~10:} approximation of $d_{\mathrm{match},k}$ by restricting to a coarser $2 \times 5$ grid with $a \in \{0.25, 0.75\}$ and $b \in \{0, 0.25, 0.5, 0.75, 1\}$.
\end{itemize}
The goal is to show that CMD provides a reliable alternative for MD~121 (which is approximately $11$ times slower), while outperforming MD~10, whose computational cost is comparable to CMD. We conduct three experiments: (i) a direct evaluation of distance quality on MNIST digits (Section~\ref{sec:mnist}), (ii) a classification task on synthetic shapes (Section~\ref{sec:synthetic}) and (iii) a classification task on chaotic attractors (Section~\ref{sec:attractors}). We conclude with a dedicated timing analysis (Section~\ref{sec:timing}), showing that the empirical computational cost aligns with the expected theoretical speed-up of CMD over the classical matching distance.

All classification experiments use $k$-nearest neighbors ($k=3$) with bottleneck distance (tolerance $\varepsilon = 0.01$). Classification is performed using a most-confident selection strategy across $H_0$ and $H_1$, and performance is evaluated via 10-fold cross-validation. Code and results are available at \url{https://doi.org/10.5281/zenodo.18644292}. Persistent homology computations were carried out using the GUDHI library~\cite{gudhi}.

\subsection{Distance quality on MNIST}
\label{sec:mnist}
The aim of this section is to quantify CMD, MD~121 and MD~10 on MNIST, 
and assess whether the discriminating power of CMD is comparable to that of MD121, which in turn is higher than that of MD10.
To keep the computational cost limited, we selected $100$ samples for each of the 10 classes, yielding $\binom{1000}{2} = 499,500$ pairs.
Let $X$ denote the $28 \times 28$ pixel grid.
We define the functions $\p_1(x) = I(x) / 255$, where $I(x) \in [0, 255]$ is the pixel intensity, and $\p_2$ as the radial distance from the image center, restricted to active pixels ($I(x) > 50$) and set to $0$ otherwise.
The function $\p_2$ is normalised to take values in $[0, 1]$.
For persistence computation, we construct a cubical complex on the $28 \times 28$ pixel grid, using the superlevel set filtration induced either by $\boldsymbol{\p}^t$ (for CMD) or $\boldsymbol{\p}^*_{\boldsymbol{{a}},\boldsymbol{b}}$ (for MD~121 and MD~10).

Figure~\ref{fig:mnist_distances} displays scatterplots comparing the three distances. Each point corresponds to an image pair, with coordinates given by the two distances being compared: for instance, in panel (a), the $x$-coordinate is CMD and the $y$-coordinate is MD~121, both computed in homological dimension 0. Points near the diagonal (red dashed line) indicate image pairs where both distances yield similar values, while points far from the diagonal indicate that one distance dominates over the other. As shown in Figure~\ref{fig:mnist_distances} (a, b), CMD and MD~121 are nearly identical in both homological dimensions. In contrast, the comparison between CMD and MD~10 (Figure~\ref{fig:mnist_distances} (c, d)) favors CMD. As a further verification, the Spearman correlation between CMD and MD~121 is $0.980$ ($H_0$) and $0.993$ ($H_1$), compared to $0.926$ ($H_0$) and $0.955$ ($H_1$) for MD~10 vs MD~121.
This confirms that CMD agrees more with the fine-grid MD than its coarser version, while maintaining a computational cost comparable to the latter.

\begin{figure*}[htbp]
    \centering
    \begin{minipage}[b]{0.24\textwidth}
        \centering
        \includegraphics[width=\textwidth]{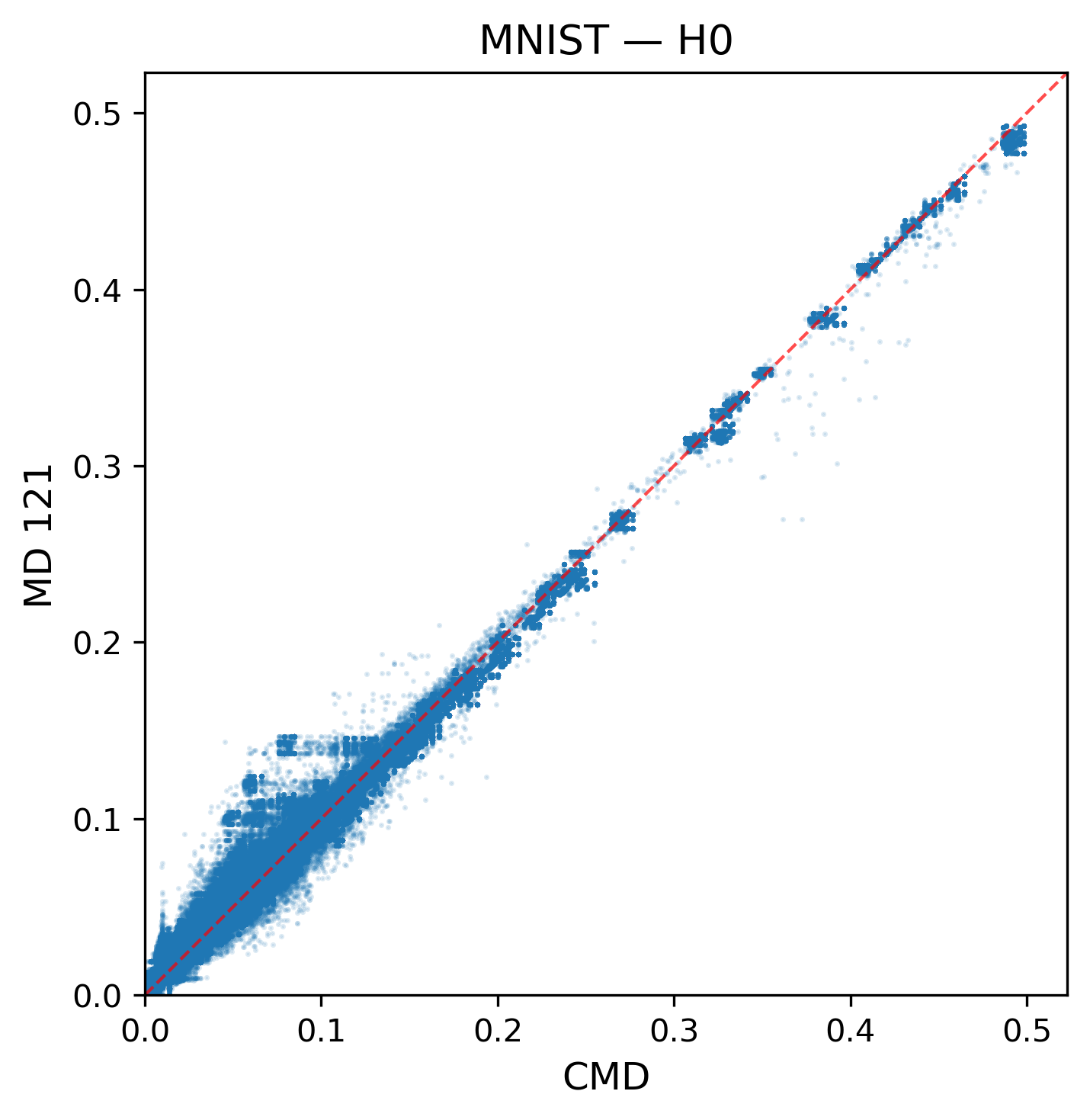}
        \\ (a) MNIST: CMD vs MD~121 for $H_0$.
    \end{minipage}
    \hfill
    \begin{minipage}[b]{0.24\textwidth}
        \centering
        \includegraphics[width=\textwidth]{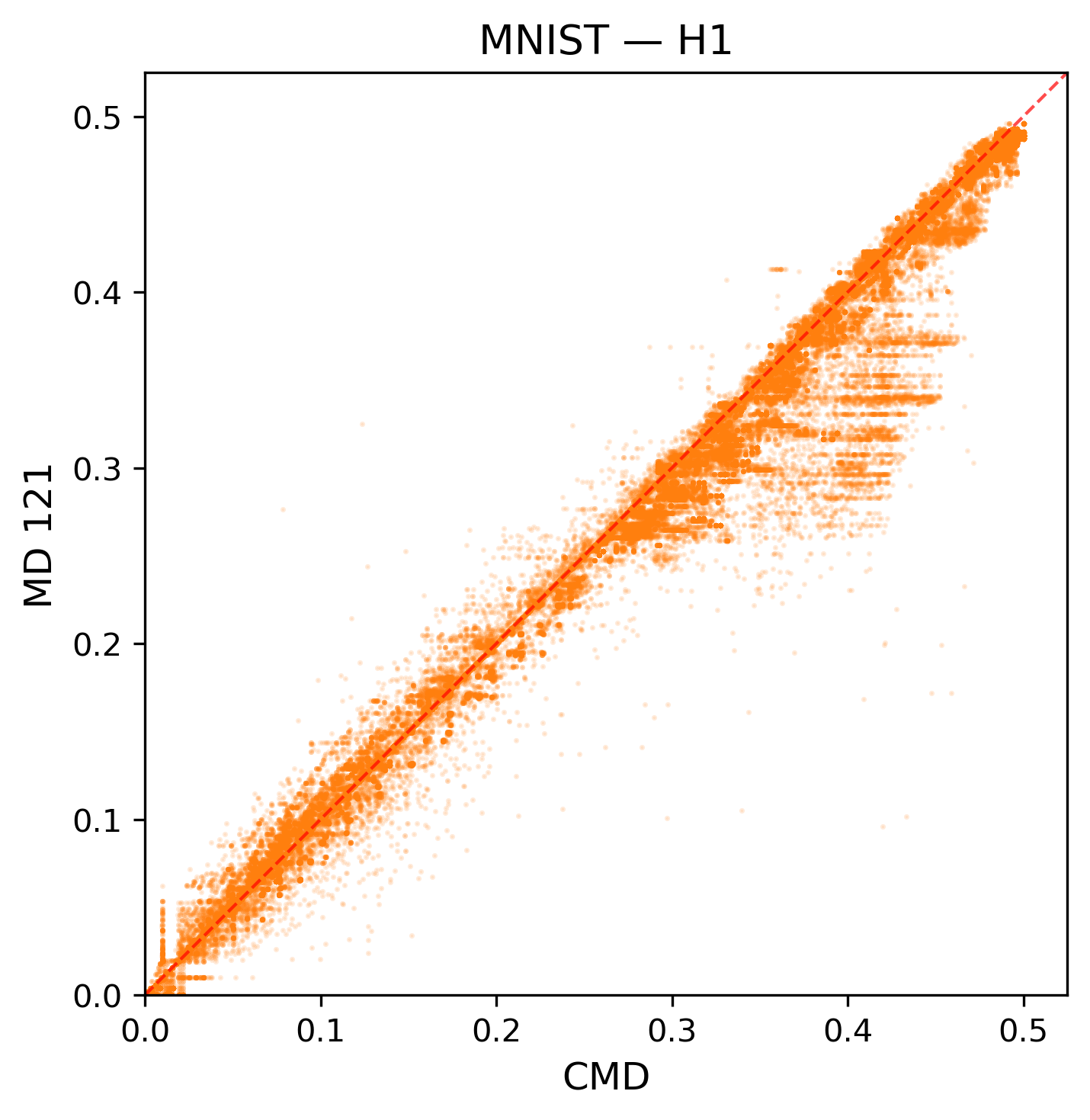}
        \\ (b) MNIST: CMD vs MD~121 for $H_1$.
    \end{minipage}
    \hfill
    \begin{minipage}[b]{0.24\textwidth}
        \centering
        \includegraphics[width=\textwidth]{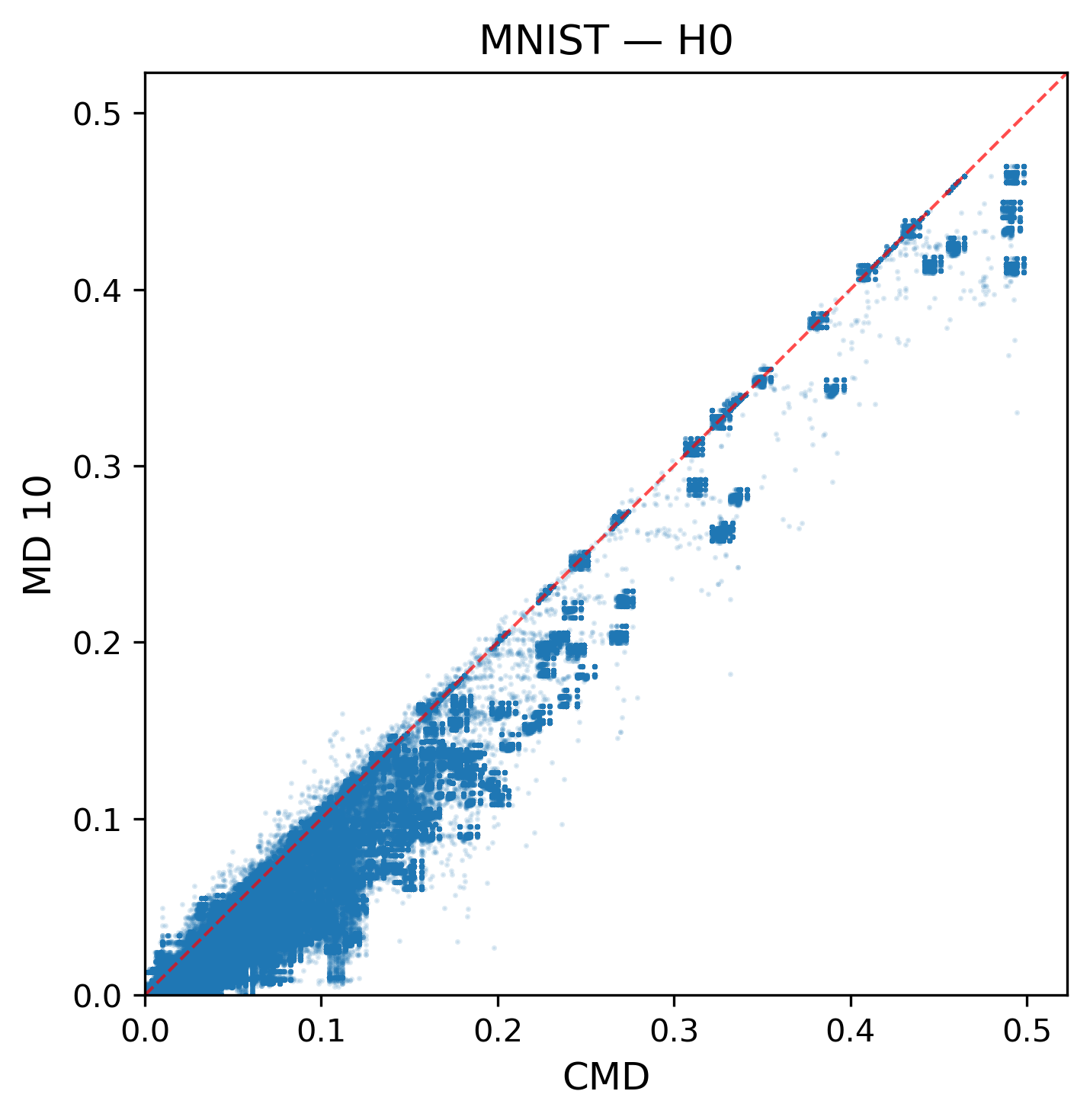}
        \\ (c) MNIST: CMD vs MD~10 for $H_0$.
    \end{minipage}
    \hfill
    \begin{minipage}[b]{0.24\textwidth}
        \centering
        \includegraphics[width=\textwidth]{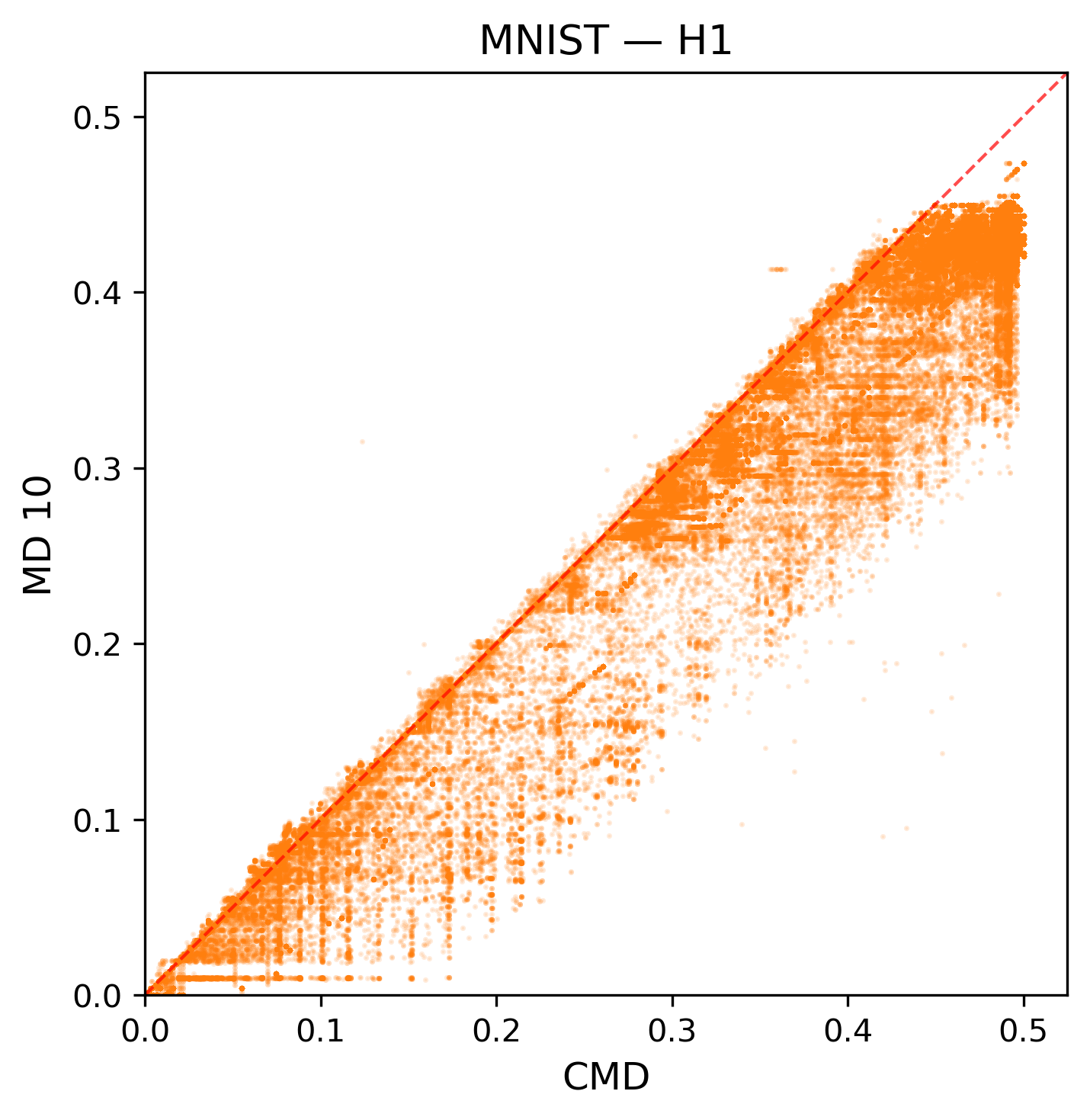}
        \\ (d) MNIST: CMD vs MD~10 for $H_1$.
    \end{minipage}

    \caption{Distance comparison for MNIST. (a,b) CMD vs MD~121 for $H_0$ and $H_1$. (c,d) CMD vs MD~10 for $H_0$ and $H_1$. Red dashed line shows $y=x$.}
    \label{fig:mnist_distances}
\end{figure*}

\subsection{Classification of synthetic shapes}
\label{sec:synthetic}
As a second experiment, we use CMD, MD~121, and MD~10 directly as classifiers in a $k$-nearest neighbors framework.
For each test sample, we compute the $k=3$ nearest neighbors in both $H_0$ and $H_1$ separately.
We then select the prediction from the homological dimension with higher confidence, defined as the proportion of the $k$ neighbors voting for the most common class.
In case of equal confidence, we use the average distance as a tiebreaker.
We generate five classes of 3D point clouds in $[0, 1]^3$: circle, sphere, torus, three clusters, and two concentric circles.
We refer to Figure~\ref{fig:synthetic_examples} for examples of point clouds for each class, colored by eccentricity.
The bifiltration uses $\varphi_1$ as the distance to the $15$-th neighbor (codensity), and $\varphi_2$ as eccentricity~\cite{BotschEurographicsSO}, both normalized to $[0, 1]$.
Persistence is computed via Alpha complex with lower-star extension.
Each class contains $40$ samples, and we test across additive Gaussian noise levels $\sigma \in \{0.00, 0.03, 0.06, 0.09, 0.12\}$ applied independently to each coordinate of the point cloud, and number of points in the point cloud $n \in \{500, 1000, 5000\}$.
As single-parameter baselines, we compute both codensity lower-star and eccentricity lower-star persistence.
Table~\ref{tab:synthetic} reports classification accuracies (mean $\pm$ standard deviation) across all combinations of sample size $n$ and noise level $\sigma$. The results are consistent across all configurations: CMD and MD~121 achieve comparable performance throughout, while MD~10 systematically underperforms. The key difference emerges at high sample size and moderate-to-high noise. At $n = 5000$, $\sigma = 0.06$: CMD achieves $97.5\%$, surpassing MD~121 ($94.2\%$), while MD~10 drops to $80.8\%$. At $\sigma = 0.12$: CMD ($78.2\%$) and MD~121 ($81.0\%$) remain comparable, while MD~10 falls to $58.0\%$. Figure~\ref{fig:synthetic_n5000} visualizes this trend for $n=5000$, showing that CMD maintains robust performance as noise increases, while MD~10 degrades significantly at $\sigma \geq 0.09$. Notably, MD~10 performs comparably to the best single-parameter baseline (codensity lower-star, see Table~\ref{tab:synthetic_1ph}), while CMD and MD~121 substantially outperform all 1-PH methods, confirming that the second filtration carries complementary discriminative information.

\begin{figure}[htbp]
    \centering
    \begin{minipage}[b]{0.19\textwidth}
        \centering
        \includegraphics[width=\textwidth]{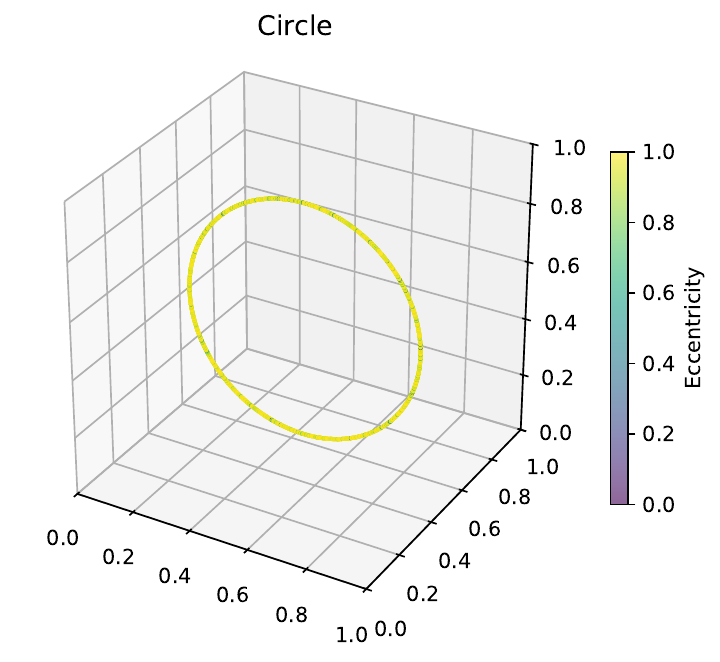}
        \\ Circle
    \end{minipage}
    \hfill
    \begin{minipage}[b]{0.19\textwidth}
        \centering
        \includegraphics[width=\textwidth]{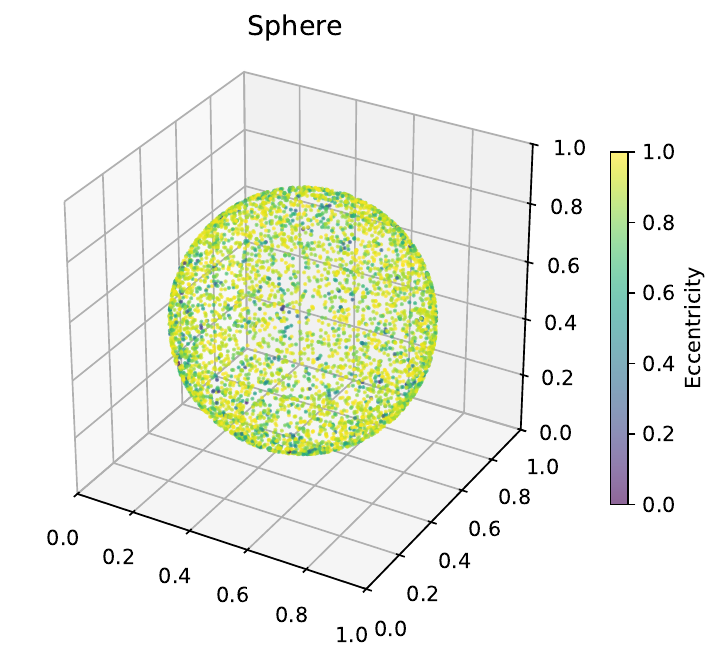}
        \\ Sphere
    \end{minipage}
    \hfill
    \begin{minipage}[b]{0.19\textwidth}
        \centering
        \includegraphics[width=\textwidth]{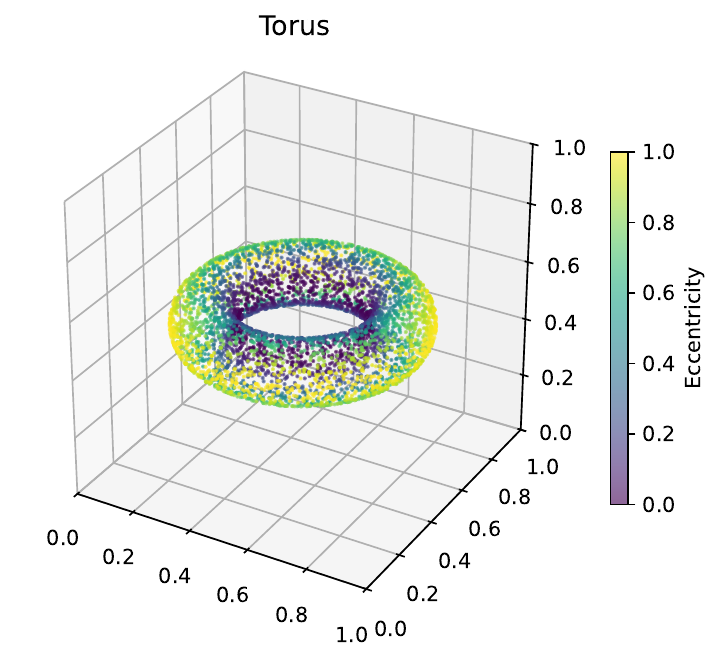}
        \\ Torus
    \end{minipage}
    \hfill
    \begin{minipage}[b]{0.19\textwidth}
        \centering
        \includegraphics[width=\textwidth]{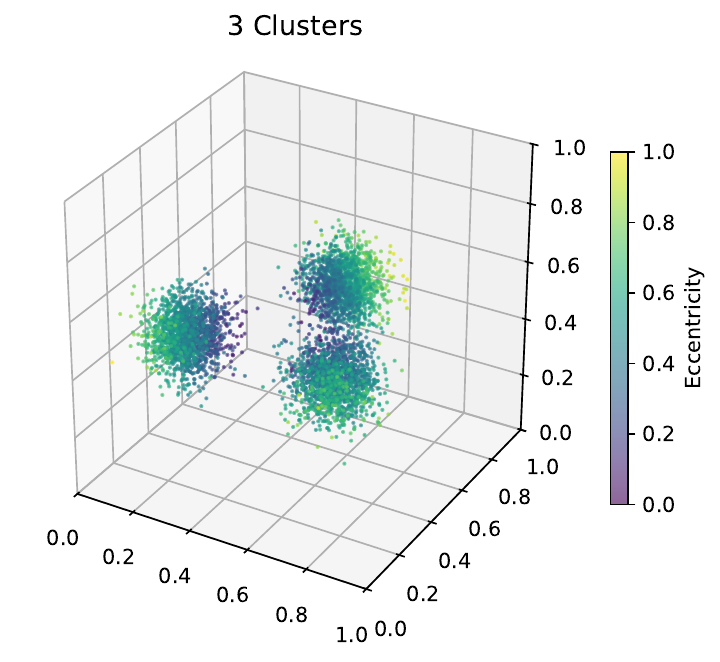}
        \\ 3 Clusters
    \end{minipage}
    \hfill
    \begin{minipage}[b]{0.19\textwidth}
        \centering
        \includegraphics[width=\textwidth]{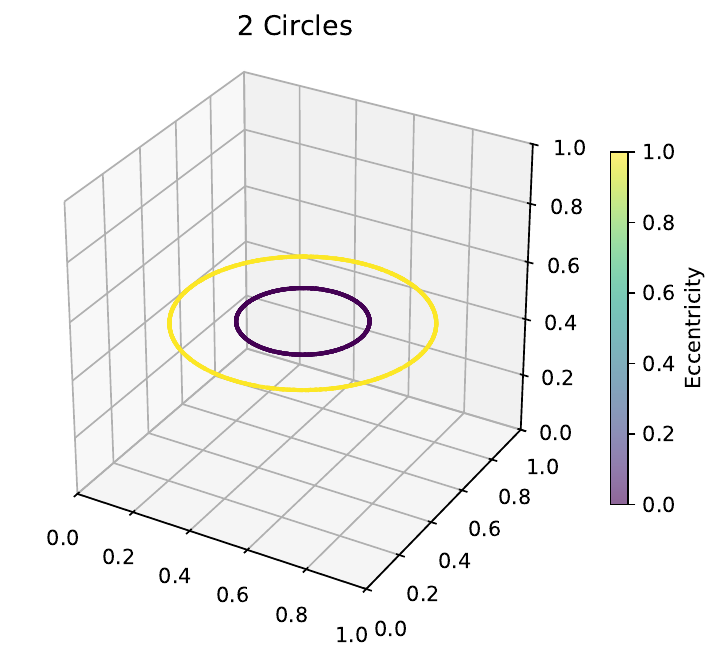}
        \\ 2 Circles
    \end{minipage}

    \caption{Examples of synthetic 3D point clouds for each class ($n=5000$, $\sigma=0.0$), colored by eccentricity.}
    \label{fig:synthetic_examples}
\end{figure}

\begin{table*}[ht]
\centering
\caption{Classification accuracy (\%) on synthetic shapes.}
\label{tab:synthetic}
\resizebox{\textwidth}{!}{%
\begin{tabular}{c|ccc|ccc|ccc}
\toprule
 & \multicolumn{3}{c|}{$n=500$} & \multicolumn{3}{c|}{$n=1000$} & \multicolumn{3}{c}{$n=5000$} \\
$\sigma$ & CMD & MD~121 & MD~10 & CMD & MD~121 & MD~10 & CMD & MD~121 & MD~10 \\
\midrule
0.00 & $91.8 \pm 5.3$ & $91.0 \pm 6.2$ & $91.3 \pm 5.6$ & $94.5 \pm 3.9$ & $92.5 \pm 6.4$ & $93.3 \pm 4.1$ & $99.3 \pm 0.8$ & $99.8 \pm 0.5$ & $99.0 \pm 1.1$ \\
0.03 & $85.2 \pm 2.8$ & $88.0 \pm 4.0$ & $84.5 \pm 3.4$ & $94.5 \pm 2.5$ & $94.2 \pm 1.5$ & $93.3 \pm 2.7$ & $99.3 \pm 0.8$ & $98.5 \pm 0.9$ & $97.2 \pm 2.2$ \\
0.06 & $76.0 \pm 3.3$ & $80.7 \pm 4.4$ & $69.5 \pm 5.1$ & $89.0 \pm 2.4$ & $92.5 \pm 3.6$ & $78.0 \pm 7.4$ & $97.5 \pm 1.9$ & $94.2 \pm 2.5$ & $80.8 \pm 4.2$ \\
0.09 & $74.0 \pm 5.9$ & $75.0 \pm 5.3$ & $70.3 \pm 1.9$ & $77.3 \pm 6.4$ & $82.3 \pm 6.3$ & $78.8 \pm 6.9$ & $88.5 \pm 2.8$ & $89.0 \pm 3.2$ & $86.3 \pm 2.6$ \\
0.12 & $69.0 \pm 4.8$ & $69.3 \pm 4.9$ & $64.0 \pm 5.7$ & $69.5 \pm 5.2$ & $74.0 \pm 5.4$ & $70.0 \pm 5.8$ & $78.2 \pm 2.8$ & $81.0 \pm 3.5$ & $58.0 \pm 5.2$ \\
\bottomrule
\end{tabular}%
}
\end{table*}

\begin{figure}[htbp]
    \centering
    \includegraphics[width=0.9\textwidth]{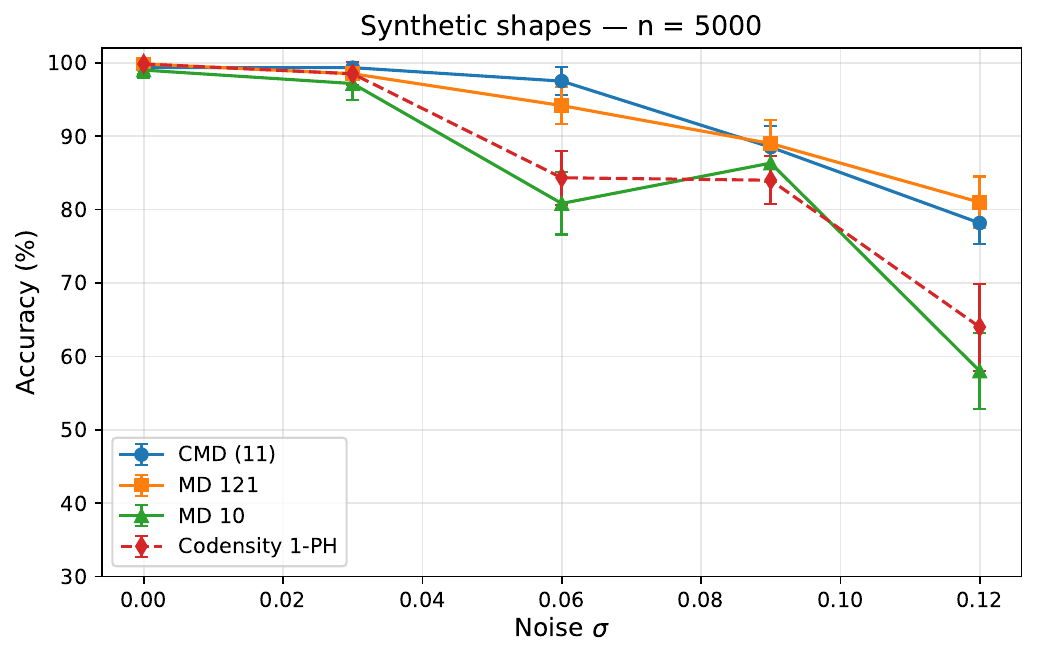}
    \caption{Classification accuracy vs noise for synthetic shapes at $n=5000$. CMD maintains robust performance across all noise levels, while MD~10 degrades significantly at high noise ($\sigma \geq 0.09$).}
    \label{fig:synthetic_n5000}
\end{figure}

\begin{table}[ht]
\centering
\caption{Single-parameter baseline accuracies (\%) on synthetic shapes.}
\label{tab:synthetic_1ph}
\resizebox{\textwidth}{!}{%
\begin{tabular}{c|cc|cc|cc}
\toprule
 & \multicolumn{2}{c|}{$n=500$} & \multicolumn{2}{c|}{$n=1000$} & \multicolumn{2}{c}{$n=5000$} \\
$\sigma$ & Codensity LS & Ecc LS & Codensity LS & Ecc LS & Codensity LS & Ecc LS \\
\midrule
0.00 & $83.7 \pm 3.7$ & $75.2 \pm 5.6$ & $93.2 \pm 3.4$ & $71.3 \pm 7.8$ & $99.8 \pm 0.5$ & $78.3 \pm 6.1$ \\
0.03 & $75.0 \pm 6.2$ & $24.2 \pm 3.7$ & $91.8 \pm 4.4$ & $23.5 \pm 4.2$ & $98.5 \pm 1.6$ & $25.0 \pm 4.0$ \\
0.06 & $77.0 \pm 4.0$ & $20.5 \pm 2.5$ & $81.7 \pm 3.9$ & $22.0 \pm 2.9$ & $84.3 \pm 3.7$ & $21.0 \pm 4.9$ \\
0.09 & $76.7 \pm 5.8$ & $19.2 \pm 3.0$ & $78.0 \pm 4.6$ & $24.8 \pm 6.8$ & $84.0 \pm 3.3$ & $28.3 \pm 8.9$ \\
0.12 & $62.5 \pm 4.8$ & $21.3 \pm 2.2$ & $69.3 \pm 3.5$ & $22.2 \pm 6.4$ & $64.0 \pm 5.9$ & $25.3 \pm 5.2$ \\
\bottomrule
\end{tabular}%
}
\end{table}

\subsection{Classification of chaotic attractors}
\label{sec:attractors}
Lastly, we classify five chaotic dynamical systems: Lorenz, Rössler, Thomas, Sprott, and Four-Wing attractors.
Chaotic attractors have been widely studied in the topological data analysis literature~\cite{10.1115/1.4055184, hussain2025topological, TEMPELMAN2020132446, 7533141}.
Each trajectory is sampled with $n = 500$ points and rescaled to $[0,1]^3$.
The bifiltration uses $\p_1 = \sqrt{\alpha}$ (Alpha complex filtration values, clipped at $0.15$ for numerical stability) and $\p_2 = $ eccentricity (lower-star extension).
Clipping avoids degenerate filtration values arising from quasi-coplanar points.
We generate 40 samples per class.
We refer to Figure~\ref{fig:attractors_examples} for visual examples of trajectories from each attractor class.
The considered approximation of the convex matching distance performs comparably in this dataset to the two approximations of the matching distance: CMD $95.2 \pm 1.9\%$, MD~121 $96.5 \pm 2.4\%$, and MD~10 $93.8 \pm 3.4\%$.
In fact, the topological signal is strong enough that even a coarse approximation of the matching distance suffices.
However, a distance quality analysis as in Section~\ref{sec:mnist} (Figure~\ref{fig:attractors_distances}) reveals that CMD and MD~121 agree more closely with each other than either does with MD~10, consistent with our previous findings.

\begin{figure}[htbp]
    \centering
    \begin{minipage}[b]{0.19\textwidth}
        \centering
        \includegraphics[width=\textwidth]{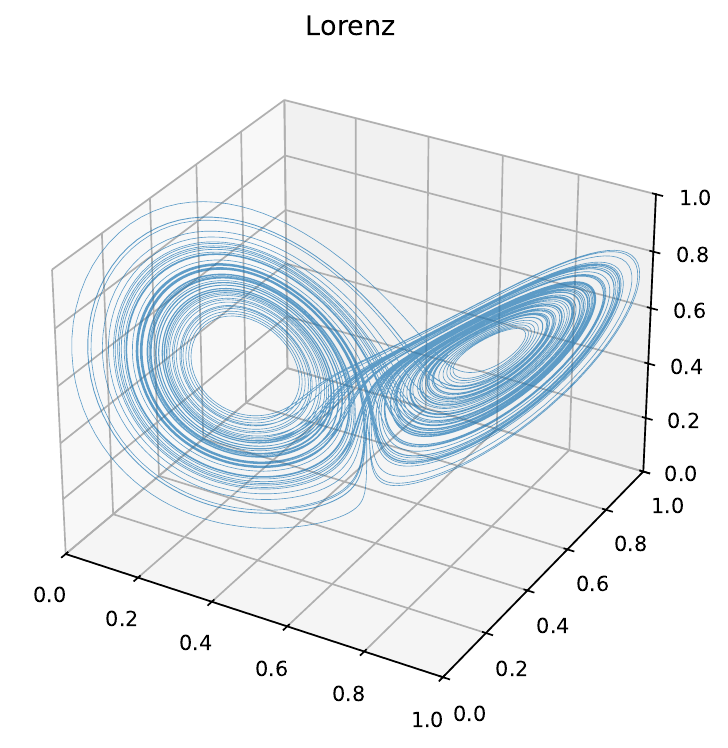}
        \\ Lorenz
    \end{minipage}
    \hfill
    \begin{minipage}[b]{0.19\textwidth}
        \centering
        \includegraphics[width=\textwidth]{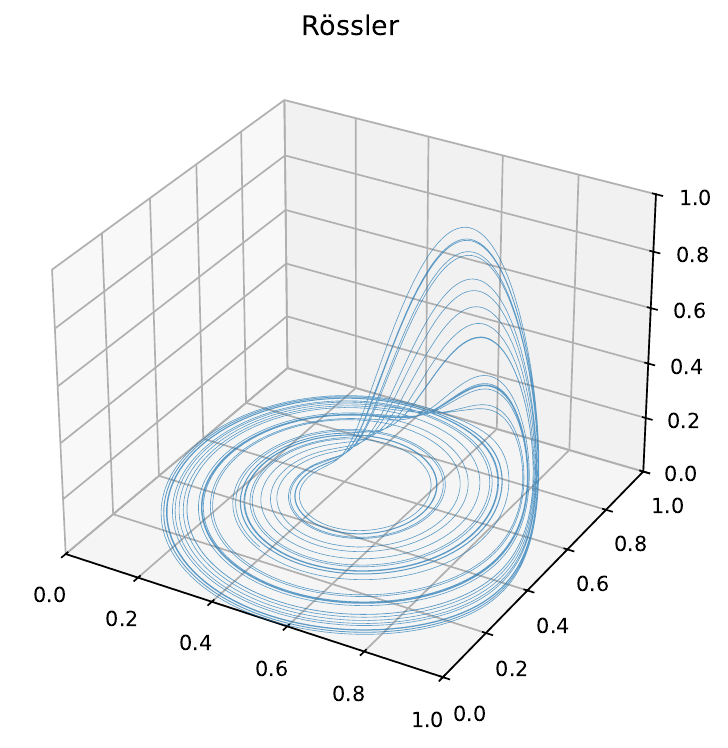}
        \\ Rössler
    \end{minipage}
    \hfill
    \begin{minipage}[b]{0.19\textwidth}
        \centering
        \includegraphics[width=\textwidth]{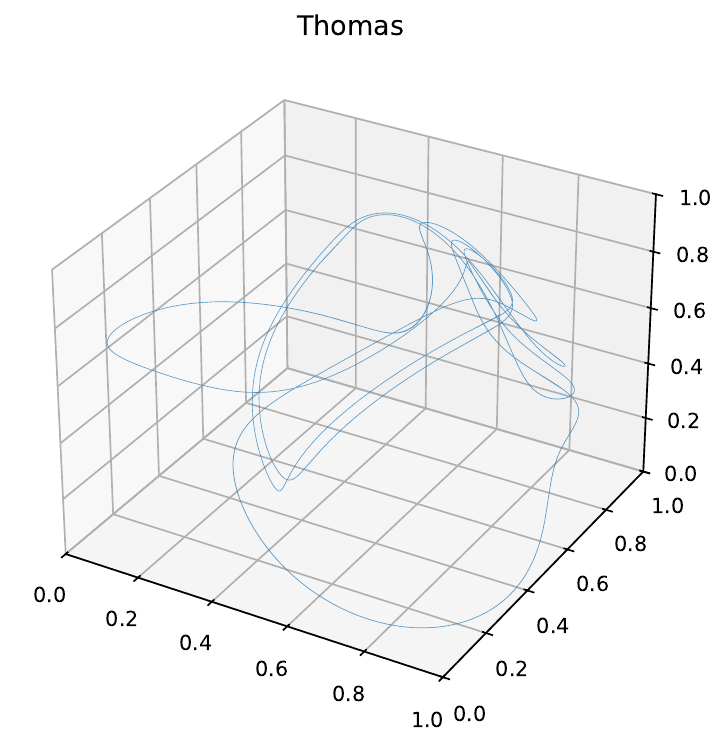}
        \\ Thomas
    \end{minipage}
    \hfill
    \begin{minipage}[b]{0.19\textwidth}
        \centering
        \includegraphics[width=\textwidth]{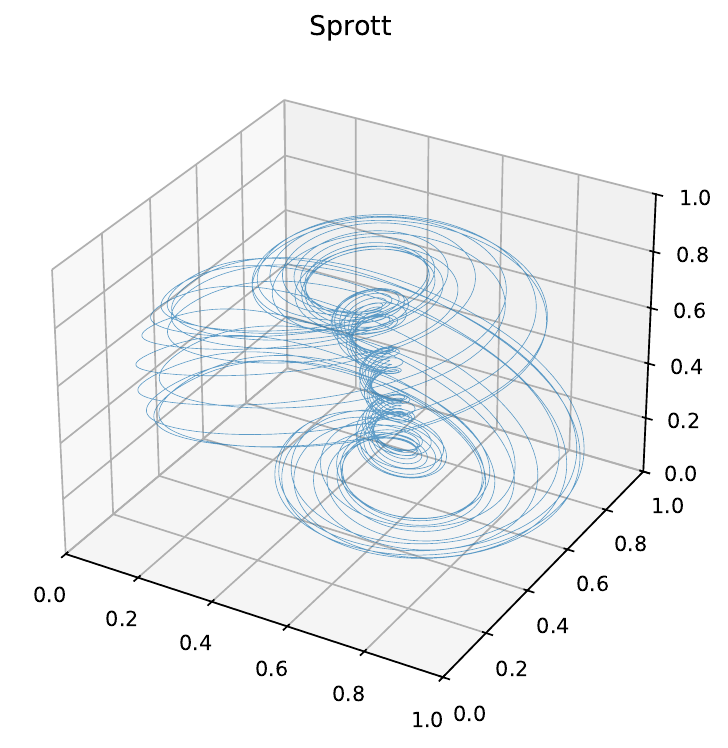}
        \\ Sprott
    \end{minipage}
    \hfill
    \begin{minipage}[b]{0.19\textwidth}
        \centering
        \includegraphics[width=\textwidth]{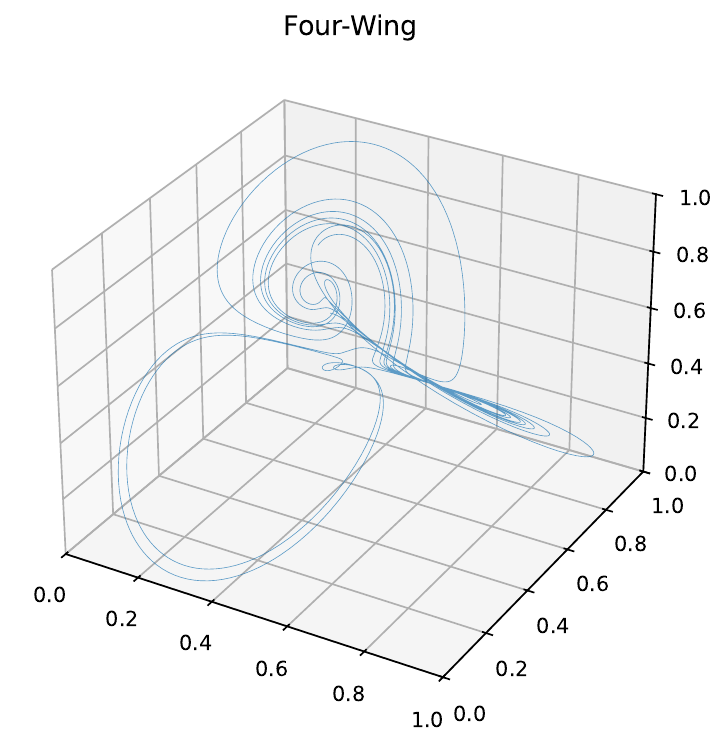}
        \\ Four-Wing
    \end{minipage}

    \caption{Examples of attractors trajectories.}
    \label{fig:attractors_examples}
\end{figure}

\begin{figure*}[htbp]
    \centering
    \begin{minipage}[b]{0.24\textwidth}
        \centering
        \includegraphics[width=\textwidth]{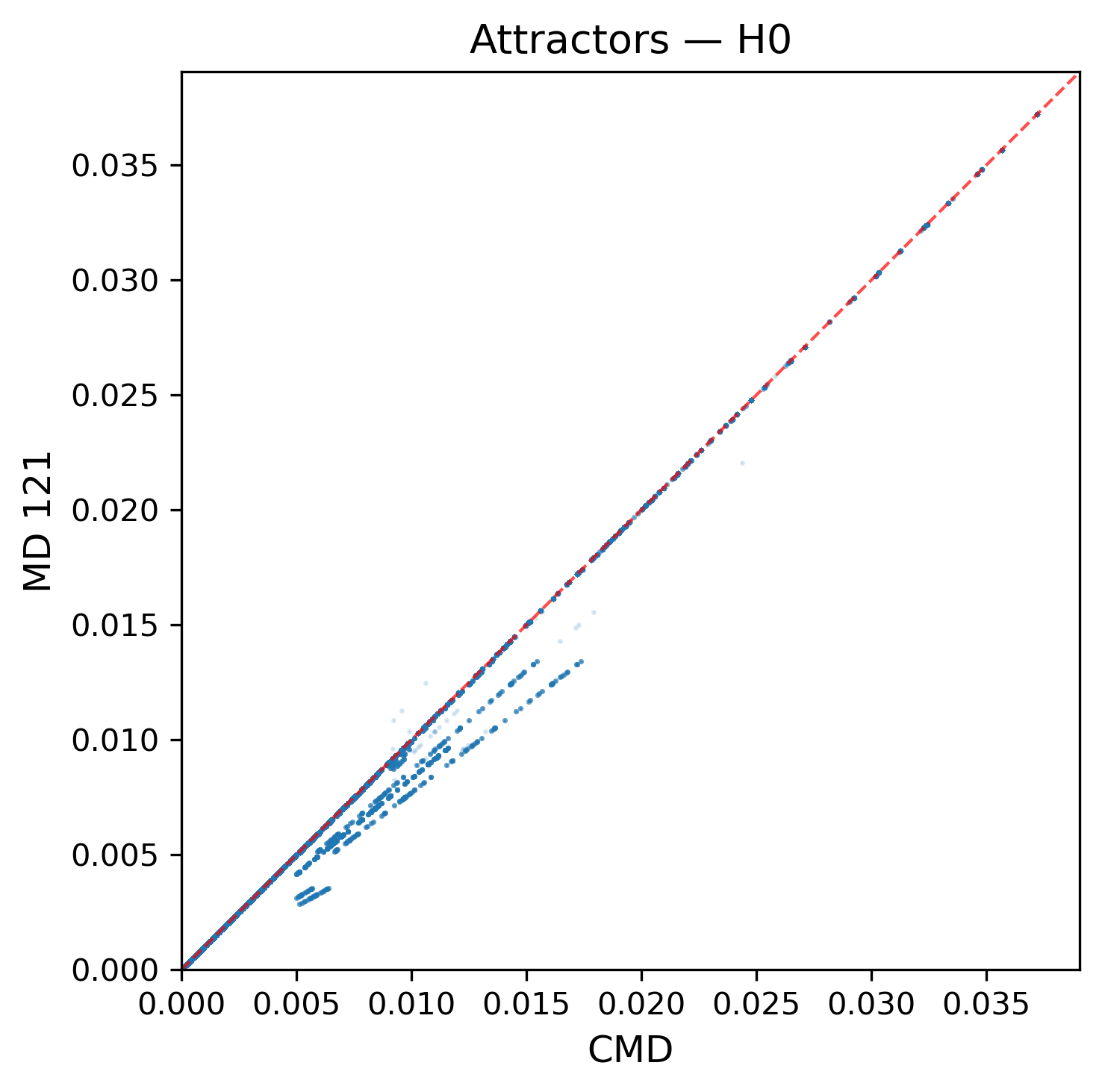}
        \\ (a) Attractors: CMD vs MD~121 for $H_0$.
    \end{minipage}
    \hfill
    \begin{minipage}[b]{0.24\textwidth}
        \centering
        \includegraphics[width=\textwidth]{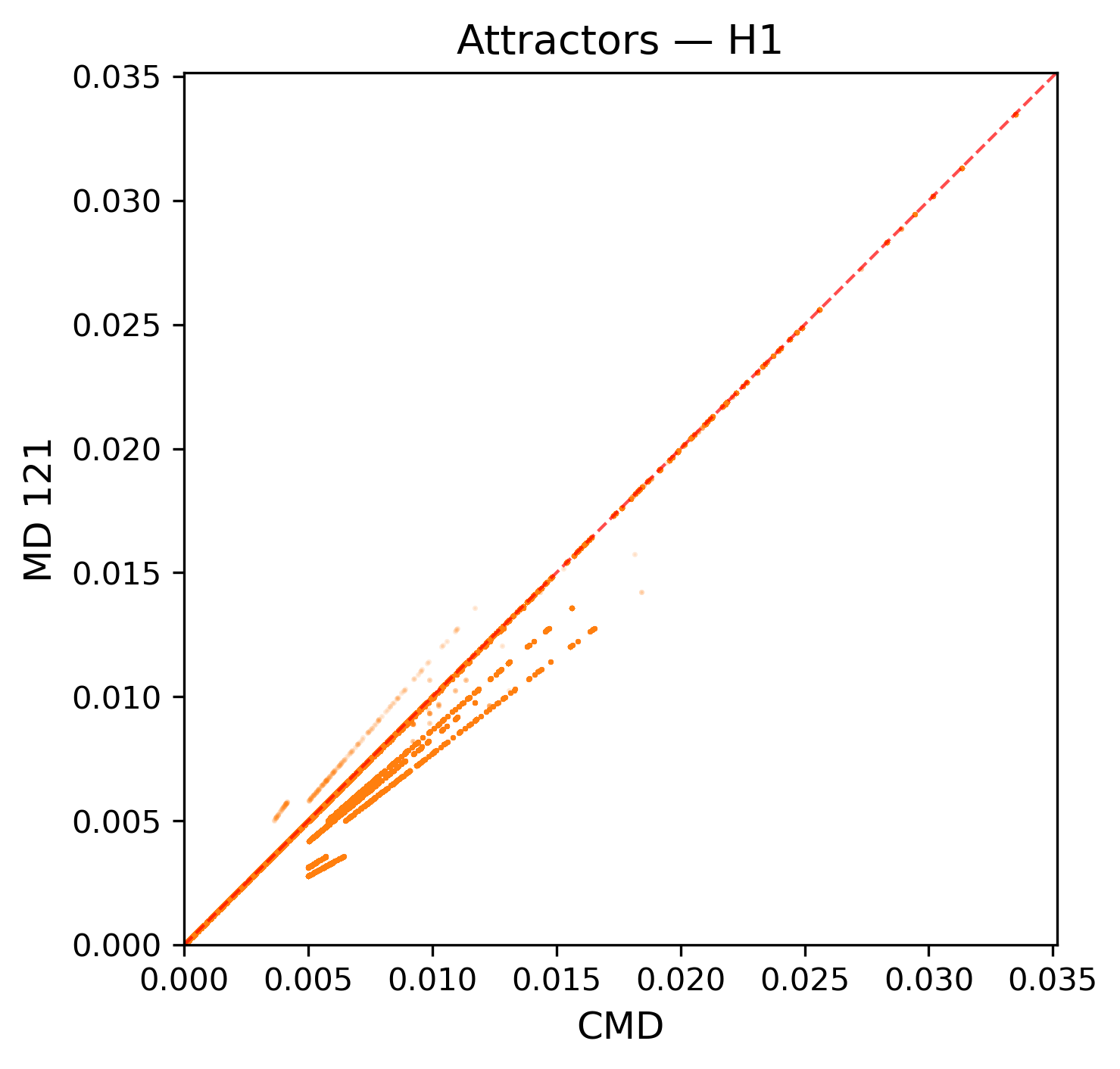}
        \\ (b) Attractors: CMD vs MD~121 for $H_1$.
    \end{minipage}
    \hfill
    \begin{minipage}[b]{0.24\textwidth}
        \centering
        \includegraphics[width=\textwidth]{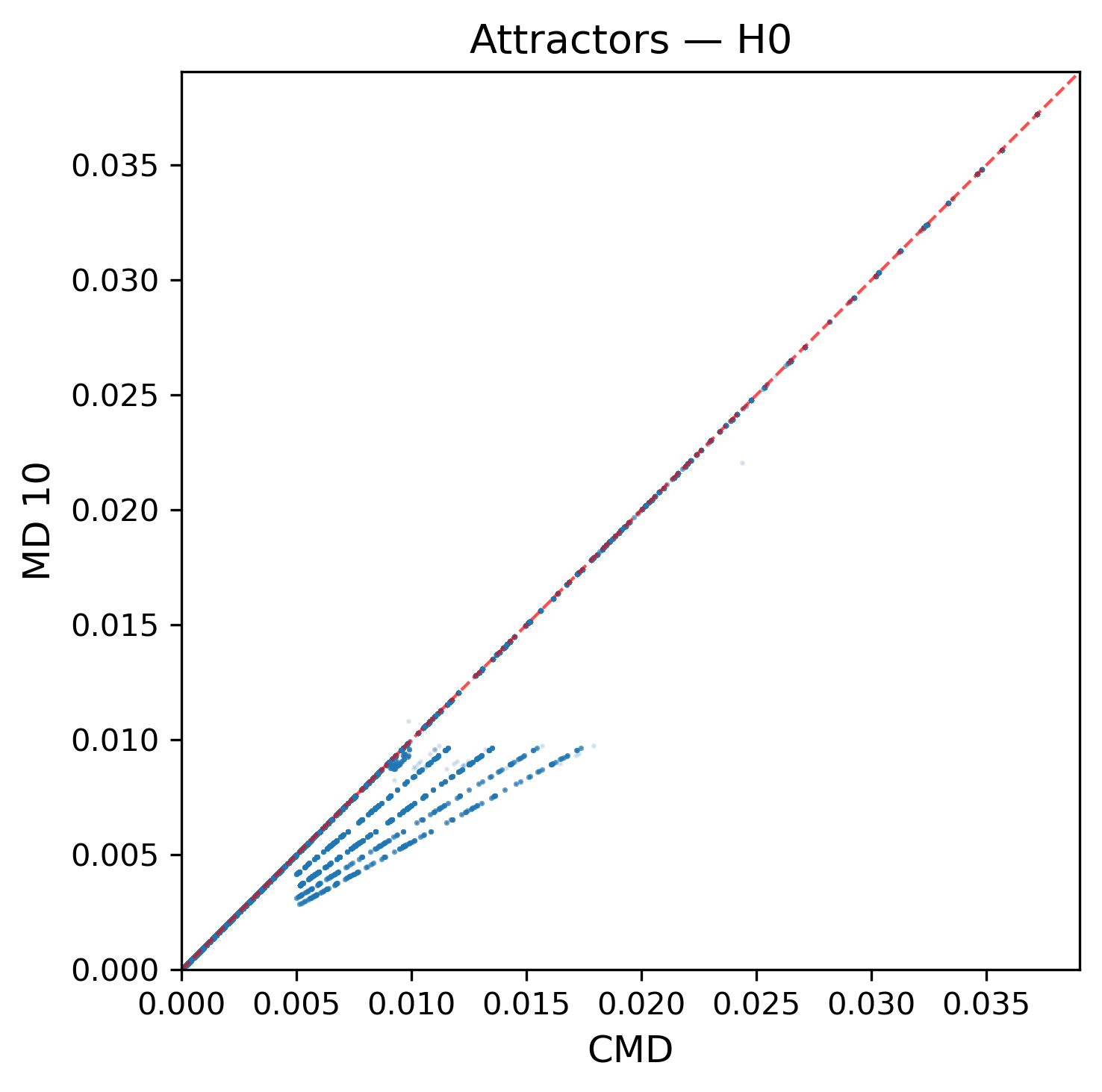}
        \\ (c) Attractors: CMD vs MD~10 for $H_0$.
    \end{minipage}
    \hfill
    \begin{minipage}[b]{0.24\textwidth}
        \centering
        \includegraphics[width=\textwidth]{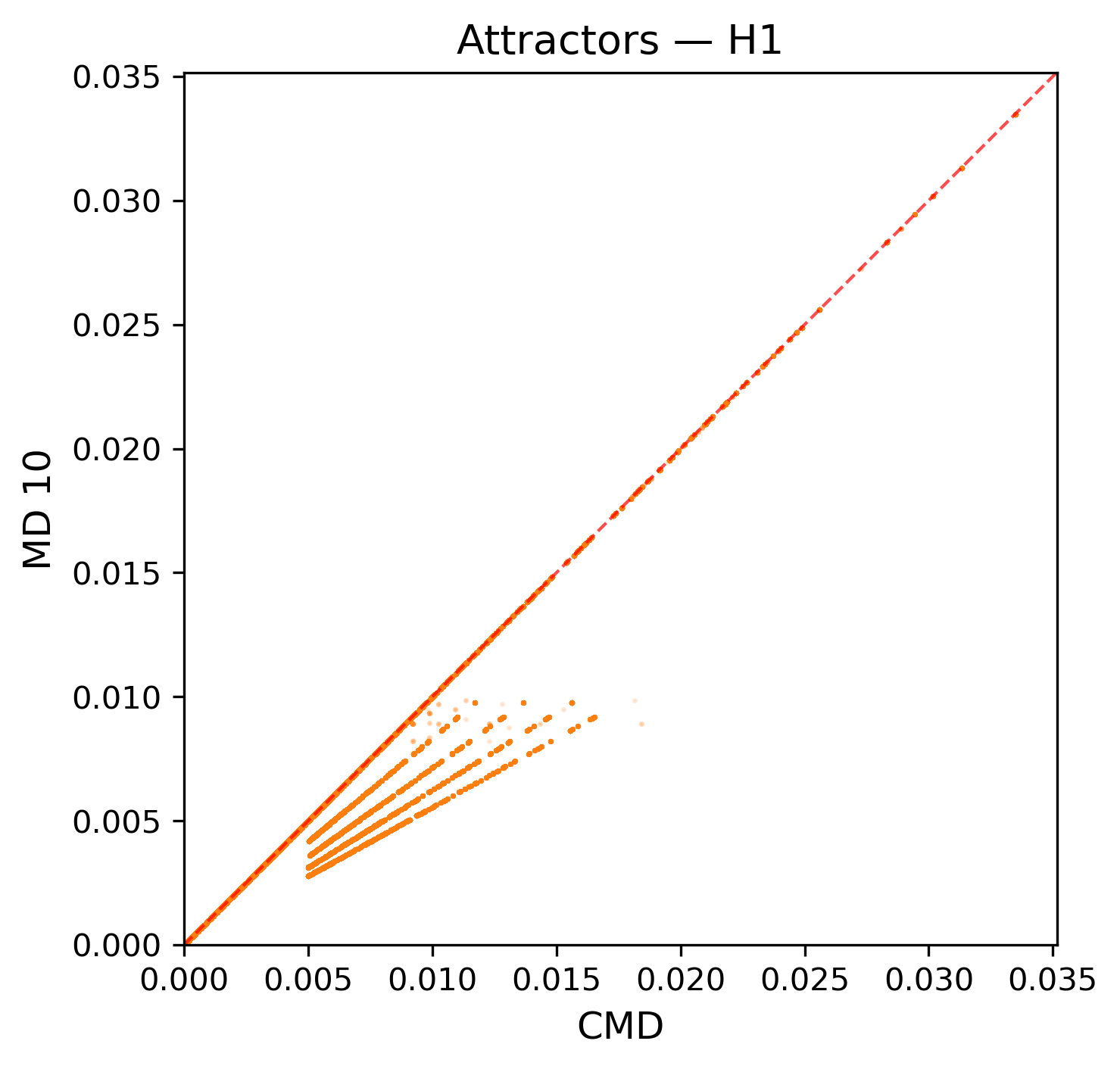}
        \\ (d) Attractors: CMD vs MD~10 for $H_1$.
    \end{minipage}

    \caption{Distance comparison for the attractors dataset. (a,b) CMD vs MD~121 for $H_0$ and $H_1$. (c,d) CMD vs MD~10 for $H_0$ and $H_1$. Red dashed line shows $y=x$.}
    \label{fig:attractors_distances}
\end{figure*}

\subsection{Computational cost}
\label{sec:timing}

We evaluate execution time on a subset of MNIST (150 images). The cost of computing the distance is dominated by persistence diagram computations and bottleneck distance evaluations. Table~\ref{tab:timing} reports timings for the three approximation strategies.

\begin{table}[ht]
\centering
\caption{Wall-clock time (seconds) for distance matrix computation on MNIST subset (150 images).}
\label{tab:timing}
\begin{tabular}{lccc}
\toprule
Method & PD computations & Bottleneck distances & Total \\
\midrule
CMD & 1.3s & 1.5s & 2.8s ($1.0\times$) \\
MD~121 & 13.2s & 12.2s & 25.4s ($9.0\times$) \\
MD~10 & 1.2s & 1.1s & 2.2s ($0.8\times$) \\
\bottomrule
\end{tabular}
\end{table}

The observed speedups of MD~121 and MD~10 over CMD is lower than the theoretical ratio due to differences in persistence diagram sizes.
The matching distance parametrisation produces on average smaller and sparser diagrams, especially for $a, b$ close to $0$ and $1$.
This reduces the cost per bottleneck distance computation. Nevertheless, CMD achieves nearly identical execution time to MD~10 while providing substantially better approximation quality.

\section{Conclusions}

In this work, we introduced the convex matching distance, 
a new pseudo-metric
to compare functions with values in $\R^2$.
If $\bo\p$ and $\bo\s$ are two such functions, the convex matching distance is
defined by maximising the bottleneck distance between the persistence diagrams of the convex combinations $\bo\p^t$ and $\bo\s^t$ of the coordinates of $\bo\p$ and $\bo\s$.
Whereas the classical matching distance requires the computation of a two-parameter family of persistence diagrams (see \cite{CeDFFeFrLa13, Landi2022, FrMoQuTo25}), our metric depends only on the one-parameter family indexed by $t\in[0,1]$.
This reduction in dimensionality, combined with the additional regularity of the operator sending $\bo\p$ to $\bo\p^t$, leads to both conceptual simplifications and computational advantages improving the effective computability of the metric.
Moreover, the smoothness preservation property of the proposed operator
allows for its incorporation into learning-based frameworks, ensuring stable and well-behaved gradients and enabling the use of standard optimization and training procedures.


Under generic regularity assumptions on the input functions,
we show in Theorem \ref{thm:main} that the convex matching distance is realised at some parameter in a distinguished subset of $[0,1]$, that we called special set.
We conjecture that this subset is finite under generic conditions on the input pair $(\bo\p,\bo\s)$.
It is not hard to find examples of pairs of functions having an infinite special set.
For example, this happens when $\mathrm{Ctr}(\bo \p, \bo \psi)$ contains two contours such that one is the translated of the other, or when a contour has zero curvature.
However, in our opinion, these phenomena occur when the pair of functions is not generic, as the functions can be infinitesimally modified locally, so not to have this properties any more.

We validated the practical effectiveness of the convex matching distance through experiments on MNIST, synthetic 3D shapes under varying noise levels, and chaotic dynamical attractors. Our results demonstrate that the convex matching distance achieves accuracy comparable to the full matching distance computed on a fine $11 \times 11$ parameter grid, while requiring approximately $10\times$ fewer persistence diagram computations. Moreover, the convex matching distance substantially outperforms coarser approximations of the matching distance, confirming that the convex parametrization effectively captures the discriminative information of the bifiltration.

Future work will focus on the effective computation of the parameters belonging to this special set as well as on the identification of the proper genericity assumptions under which such a set is finite.
Moreover, we aim to weaken the assumptions under which we proved Theorem \ref{thm:main} and to prove their genericity.

Finally, as mentioned in the introduction the operator that given a parameter $t$ maps each function $\boldsymbol{\varphi}$ to the function $\boldsymbol{\varphi}^{t}$ is just one of the infinitely many possibilities for assigning a parametric family of functions to $\bo \p$. In the near future, we intend to investigate alternative operators and evaluate the advantages and drawbacks associated with their adoption.






\subsection*{Acknowledgements}
F.C. was funded by Inria.
P.F. was supported by GNSAGA-INdAM, CNIT National Laboratory WiLab, and the WiLab-Huawei Joint Innovation Center.
U.F. acknowledges the Italian National Biodiversity Future Center (NBFC) - National Recovery and Resilience Plan (NRRP) funded by the EU-NextGenerationEU (project code CN 00000033) and the Innovation Ecosystem Robotics and AI for Socio-economic Empowerment (RAISE) - National Recovery and Resilience Plan (NRRP) funded by the EU-NextGenerationEU (project code ECS 00000035).
E.M.G. was supported by GNAMPA-INdAM and GNAMPA Project CUP E53C22001930001.
N.Q. was supported by CNIT National Laboratory WiLab, and the WiLab-Huawei Joint Innovation Center.
S.S. acknowledges the GNSAGA-INdAM and the MUR Excellence Department Project MatMod@TOV awarded to the Department of Mathematics, University of Rome Tor Vergata, CUP E83C23000330006.
F.T. was funded by the Knut and Alice Wallenberg Foundations and the WASP Postdoctoral Scholarship Program.

\bibliographystyle{plainurl}
\bibliography{references}

\end{document}